\documentclass[11pt]{article}
\makeatletter
\newcommand{\labitem}[2]{%
\def\@itemlabel{\text{#1}}
\item
\def\@currentlabel{#1}\label{#2}}
\makeatother
\usepackage{setspace}

\usepackage{chngcntr}

\usepackage{booktabs}

\usepackage{amsmath}
\usepackage{extsizes}

\usepackage{amsthm}
\usepackage{epsfig}
\usepackage{amsfonts}
\newtheorem{theorem}{Theorem}
\usepackage[left=3.7cm, right=3.7cm,bottom = 3cm, top=2.9cm]{geometry}
\usepackage{graphicx}
\usepackage{courier}
\usepackage{colortbl}
\usepackage{multirow}
\RequirePackage[OT1]{fontenc}
\RequirePackage{amsthm,amsmath}

\usepackage{graphicx} 
\usepackage{amsmath}
\usepackage{amssymb}
\usepackage{float}
\usepackage{enumerate}

\usepackage{booktabs}
\usepackage{enumerate, marvosym,enumitem}

\usepackage{array, xcolor, lipsum, bibentry}
\usepackage[round]{natbib}

\newcommand{\vertiii}[1]{{\left\vert\kern-0.25ex\left\vert\kern-0.25ex\left\vert #1
    \right\vert\kern-0.25ex\right\vert\kern-0.25ex\right\vert}}

\newtheorem{cor}{Corollary}

\newtheorem{assump}{Assumption}

\newtheorem{condition}{Condition}

\newtheorem{lema}{Lemma}

	\providecommand{\keywords}[1]{\textbf{\textbf{Keywords:}} #1}

\begin{document}

\title{Honest confidence regions and optimality in high-dimensional precision matrix estimation
}
%
%
\author{Jana Jankov\'a       \and
        Sara van de Geer 
}
%

%
%
%
\maketitle

\begin{abstract}
\noindent
We propose methodology for estimation of sparse precision matrices and statistical inference for their low-dimensional parameters in a high-dimensional setting where the number of parameters $p$ can be much larger than the sample size. We show that the novel estimator achieves minimax rates in supremum norm and the low-dimensional components of the estimator have a Gaussian limiting distribution. These results hold uniformly over the class of precision matrices with row sparsity of small order $\sqrt{n}/\log p$ and spectrum uniformly bounded, under a sub-Gaussian tail assumption on the margins of the true underlying distribution. Consequently, our results lead to uniformly valid confidence regions for low-dimensional parameters of the precision matrix. Thresholding the estimator leads to variable selection without imposing irrepresentability conditions. The performance of the method is demonstrated in a simulation study and on real data.

\vskip 0.2cm
\noindent
\vskip 0.2cm

\noindent
\begin{subclass}{62J07 \and 62F12  }
\end{subclass}
\end{abstract}
\noindent

\section{Introduction}
We consider the problem of estimation of the inverse covariance matrix in a high-dimensional setting, where the number of parameters $p$ can significantly exceed the sample size $n$. 
Suppose that we are given an $n\times p$ design matrix $\mathbf X$, where the rows of $\mathbf X$ are $p$-dimensional i.i.d. random vectors from an unknown distribution  with mean zero and covariance matrix $\Sigma_0\in\mathbb R^{p\times p}.$ 
We denote the precision matrix by $\Theta_0:=\Sigma_0^{-1}$, assuming the inverse of $\Sigma_0$ exists.
\\
The problem of estimating the precision matrix arises in a wide range of applications. Precision matrix estimation in particular plays an important role in graphical models that have become a popular tool for representing dependencies within large sets of variables. 
Suppose that we associate the variables $X_1 ,\dots,X_p$ with the vertex set $\mathcal V =
\{1,\dots,p\}$ of an undirected graph $G = (\mathcal V,\mathcal E)$ with an edge set $\mathcal E$. A graphical
model $G$ represents the conditional dependence relationships between the variables, namely every pair of variables not contained in the edge set is conditionally independent given all remaining variables. If the vector $(X_1 ,\dots,X_p)$ is normally distributed, each
edge corresponds to a non-zero entry in the precision matrix (\citet{lauritzen}). 
Practical examples of applications of graphical modeling include modeling of brain connectivity
based on FMRI brain analysis \citep{fmri}, genetic networks, financial data processing, social network analysis
and climate data analysis.\\
A lot of work has been done on methodology for \emph{point} estimation of precision matrices. We discuss some of the approaches below, but a selected list of papers includes for instance  \citet{buhlmann,glasso,bickel2,yuan2,cai,sunzhang}.
A common approach assumes that the precision matrix is sufficiently sparse and employs the $\ell_1$-penalty to induce a sparse estimator. The main goal of these works is to show that, under some regularity conditions, the sparse estimator  behaves almost as the oracle estimator that has the  knowledge of  the true sparsity pattern.\\
Our primary interest in this paper lies not in {point} estimation, but we aim to quantify uncertainty of estimation by providing \textit{interval} estimates for the entries of the precision matrix. 
The challenge of this problem arises since asymptotics of regularized estimators which are the main tool in high-dimensional estimation is not easily tractable \citep{knight2000}, as opposed to the classical setting when the dimension of the unknown parameter is fixed.


\subsection{Overview of related work} 
Methodology for inference in high-dimensional models has been mostly studied in the context of linear and generalized linear regression models. 
From the work on linear regression models, we mention the paper by \cite{zhang} where a semi-parametric projection approach using the Lasso methodology \citep{lasso} was proposed,
which was further developed and studied in \cite{vdgeer13}. The approach leads to asymptotically normal estimation of the regression coefficients and an extension of the method to generalized linear models is given in \cite{vdgeer13}.
 The method requires sparsity of small order $\sqrt{n}/\log p$ in the high-dimensional parameter vector and uses $\ell_1$-norm error bound of the Lasso.
Further alternative methods for inference in the linear model have been proposed and studied in \cite{stanford1}, \cite{belloni1} and bootstrapping approach was suggested in \cite{bootstrap1}, \cite{bootstrap2}.
\\
Other lines of work on inference for high-dimensional models suggest post-model selection procedures, where in the first step a regularized estimator is used for model selection and in the second step e.g. a maximum likelihood estimator is applied on the selected model.  
In the linear model, simple post-model selection methods have been proposed e.g. in \cite{gausslasso}, \cite{candes2007}. 
These approaches are however only guaranteed to work under irrepresentability 
and beta-min conditions (see \cite{hds}). Especially in view of inference, beta-min conditions which assume that the non-zero parameters are sufficiently large in absolute value, should be avoided. 
\\
\noindent
In this paper, we consider estimation of precision matrices, which is a problem related to linear regression, however, it is a non-linear problem and thus it requires a more involved treatment.  
One approach to precision matrix estimation is based on regularization of the maximum likelihood in terms of the $\ell_1$-penalty. This approach is typically referred to as the graphical Lasso, and  has been studied in detail in several papers, see e.g. \cite{glasso}, \cite{rothman}, \cite{ravikumar} and \cite{yuan}. 
Another common approach to precision matrix estimation is based on projections. This approach reduces the problem to a series of regression problems and estimates each column of the precision matrix using a Lasso estimator or Dantzig selector \citep{candes2007}. The idea was first introduced in \cite{buhlmann} as neighbourhood selection for Gaussian graphical models and further studied in \cite{yuan2}, \cite{cai} and \cite{sunzhang}. 
\\
Methodology leading to statistical inference for the precision matrix has been studied only recently. The work \cite{zhou} proposes to use a more involved variation of the  regression approach to obtain an estimator which leads to statistical inference. This approach leads to an estimator of the precision matrix which is elementwise asymptotically normal, under row sparsity of order $\sqrt{n}/\log p$, bounded spectrum of the true precision matrix and Gaussian distribution of the sample.   
The paper \cite{jvdgeer14} proposes a method for statistical inference based on the graphical Lasso. The work introduces a de-sparsified estimator based on the graphical Lasso, which is also shown to be elementwise asymptotically normal. 
%
\subsection{Contributions and outline}
We propose methodology leading to honest confidence intervals and testing for low-dimensional parameters of the precision matrix, without  requiring irrepresentability conditions or beta-min conditions to hold.
Our work is motivated by the semi-parametric approach in \cite{vdgeer13} and is  a follow-up of the work \cite{jvdgeer14}. 
Compared to the previous work on statistical inference for precision matrices, this methodology  has several advantages. 
Firstly, the estimator 
we propose is a simple modification of the nodewise Lasso estimator proposed in \cite{buhlmann}. Hence the estimator is easy to implement and efficient solutions are available on the computational side. 
Secondly, the novel estimator enjoys a range of optimality properties and leads to statistical inference under mild conditions. 
 Firstly, the asymptotic distribution of low-dimensional components of the estimator is shown to be Gaussian.
This holds uniformly over the class of precision matrices with row sparsity of order $o(\sqrt{n}/\log p)$, spectrum uniformly bounded in $n$ and sub-Gaussian margins of the underlying distribution.  
This results in honest confidence regions \citep{li} for low-dimensional parameters.
The proposed estimator achieves rate optimality as shown in Section \ref{subsec:trates}. 
Moreover, the de-sparsified estimator may be thresholded to guarantee  variable selection without imposing irrepresentable
conditions. 
The computational cost of the method is order $\mathcal O(p)$ Lasso regressions for estimation of all parameters and two Lasso regressions for a single parameter. 
\\
\\
The paper is organized as follows.  
Section \ref{subsec:est} introduces the methodology. Section \ref{subsec:main} contains the main theoretical results for estimation and inference and in Section \ref{subsec:varsel} the suggested method is applied to variable selection. 
Section \ref{subsec:comparison} provides a comparison with related work.
Section \ref{subsec:sim} illustrates the theoretical results in a simulation study.
In Section \ref{subsec:real}, we analyze two real datasets and apply our method to variable selection. Section \ref{subsec:conclusions} contains a brief summary of the results.  Finally, the proofs were deferred to the Supplementary material. 
\\
\\
\noindent
\textit{Notation.} For a vector $x=(x_1,\dots,x_d)\in\mathbb R^d$ and $p\in (0,\infty]$ we use $\|x\|_p$ to denote the $p-$norm of $x$ in the classical sense. We denote $\|x\|_0 = |\{i:x_i\not =0\}|.$ For a matrix $A\in \mathbb R^{d\times d}$ we use the notations $\vertiii{ A}_\infty=\max_{i} \|e_i^T A\|_1$, $\vertiii{A}_1= \vertiii{ A^T}_\infty$ and $\|A\|_\infty=\max_{i,j}|A_{ij}|$.
The symbol $\text{vec}(A)$ denotes the vectorized version of a matrix $A$ obtained by stacking the rows of $A$ on each other.
By $e_i$ we denote a $p$-dimensional vector of zeros with one at position $i$.
For real sequences $f_n,g_n$, we write $f_n=O(g_n)$ if $|f_n| \leq C |g_n|$ for some $C>0$ independent of $n$ and all $n>C.$ 
We write 
$f_n\asymp g_n$ if both $f_n = \mathcal O(g_n)$ and $1/f_n =\mathcal O(1/g_n)$ hold. 
Finally, $f_n=o(g_n)$ if $\lim_{n\rightarrow \infty} f_n/g_n =0.$ Furthermore, for a sequence of random variables $x_n$ we write 
$x_n=\mathcal O_{\mathbb P}(1)$ if $x_n $ is bounded in probability and we write $x_n=\mathcal O_{\mathbb P}(r_n)$ if $x_n/r_n = \mathcal O_{\mathbb P}(1).$ We write $x_n=o_{\mathbb P}(1)$ if $x_n$ converges in probability to zero. \\
Let $\rightsquigarrow$ denote the convergence in distribution and $\stackrel{P}{\rightarrow }$ the convergence in probability.
Let $\Phi$ denote the cumulative distribution function of a standard normal random variable. 
By $\Lambda_{\min}( A)$ and $\Lambda_{\max}(A)$ we denote the minimum and maximum eigenvalue of $A$, respectively. 
Let $a \vee b$, $a \wedge b$ denote $\max(a,b),$ $\min(a,b),$ respectively. 
We use letters $C,c$ to denotes universal constants. These are used in the proofs repeatedly to denote possibly different  constants.

\section{De-sparsified nodewise Lasso}
\label{subsec:est}
Our methodology is a simple modification of the nodewise Lasso estimator proposed in \cite{buhlmann}. 
The idea is to remove the bias term which arises in the nodewise Lasso estimator due to $\ell_1$-penalty regularization.  
This approach in inspired by literature on semiparametric statistics
\cite{bickelbook,vdv}.
 We note several papers have used this idea in the context of high-dimensional sparse estimation, see \cite{zhang,vdgeer13,vdgeer14,stanford1,jvdgeer14}.\vskip 0.2cm
We first summarize the nodewise Lasso method introduced in \cite{buhlmann} and discuss some of its properties. This method estimates an unknown precision matrix using the idea of projections to approximately invert the sample covariance matrix.
For each $j=1,\dots,p$ we define  
the vector $ \gamma_j=\{\gamma_{j,k},k\not = j\}$ as follows
\begin{equation}\label{true}
\gamma_j := \text{arg}\min_{\gamma\in\mathbb R^{p-1}} \mathbb E\|X_j - \mathbf X_{-j}\gamma\|_2^2/n
\end{equation}
and denote $\eta_j := X_j-\mathbf X_{-j}\gamma_j$  and the noise level by $\tau_j^2 = \mathbb E\eta_j^T \eta_j/n.$
We define the column vector $\Gamma_j:= (-\gamma_{j,1},\dots, -\gamma_{j,j-1},1, -\gamma_{j,j+1} ,\dots,-\gamma_{j,p})^T.$
Then one may  show
\begin{equation}\label{idnr}
\Theta_0 = (\Theta^0_1,\dots,\Theta^0_p)=(\Gamma_1/\tau_1^2,\dots,\Gamma_p/\tau_p^2),
\end{equation}
where $\Theta^0_j$ is the $j$-th column of $ \Theta_0.$ Hence the precision matrix $\Theta_0$ may be recovered from the partial correlations $\gamma_{j,k} $ and from the noise level $\tau_j^2.$ In our problem, we are only given the design matrix $\mathbf X$. The idea of nodewise Lasso is to estimate the partial correlations and the noise levels by doing a projection of every column of the design matrix on all the remaining columns. In low-dimensional settings, this procedure would simply recover the sample covariance matrix $\mathbf X^T \mathbf X/n.$ However, due to the high-dimensionality of our setting, the matrix $\mathbf X^T \mathbf X/n$ is not invertible and we can only do approximate projections. If we assume sparsity in the precision matrix (and thus also in the partial correlations), this idea can be effectively carried out using the Lasso.   Hence, for each $j=1,\dots,p$ define the estimators of the regression coefficients, $\hat\gamma_j =\{\hat\gamma_{j,k},k=1,\dots,p,j\not = k\} \in \mathbb R^{p-1}$, as follows
\begin{equation}\label{NR}
\hat\gamma_j := \text{arg}\min_{\gamma\in\mathbb R^{p-1}} \|X_j - \mathbf X_{-j}\gamma\|_2^2/n + 2\lambda_j \|\gamma\|_1.
\end{equation}
We further define the column vectors
$$\hat{\Gamma}_j := (-\hat\gamma_{j,1},\dots, -\hat\gamma_{j,j-1},1, -\hat\gamma_{j,j+1} ,\dots,-\hat\gamma_{j,p})^T,$$
and  estimators of the noise level 
$$\hat\tau_j^2 := \|X_j - \mathbf X_{-j}\hat\gamma_j\|_2^2/n + \lambda_j \|\hat\gamma_j\|_1,$$
for $j=1,\dots,p$.
Finally, we define the $j$-th column of the nodewise Lasso estimator $\hat\Theta$ as
\begin{equation}\label{nrdef}
\hat{\Theta}_j := \hat{\Gamma}_j/\hat\tau_j^2.
\end{equation}

\noindent
The estimator $\hat \Theta_j$ of the precision matrix was studied in several papers (following \cite{buhlmann}) and has been shown to enjoy oracle properties under mild conditions on the model. These conditions  include bounded spectrum of the precision matrix, row sparsity of small order $n/\log p$ and a sub-Gaussian distribution of the rows of $\mathbf X$ (alternatively to sub-Gaussianity, one may assume that the covariates are bounded as in \cite{vdgeer13}).  
Our approach uses the nodewise Lasso estimator as an initial estimator. The next step involves de-biasing or de-sparsifying, which may be viewed 
as one step using the Newton-Raphson scheme for numerical optimization. 
This is equivalent to ``inverting'' the Karush-Kuhn-Tucker (KKT) conditions by the  inverse of the Fisher information as in \cite{vdgeer13}. The challenge then also comes from the need to estimate the Fisher information matrix which is a $p^2\times p^2$ matrix. We show that the estimator $\hat\Theta$ can be used in a certain way to create a surrogate of the inverse Fisher information matrix. 
Since the estimator $\hat\Theta_j$ can be characterized by its KKT conditions, it is convenient to work with these conditions to derive the new de-sparsified estimator.
Consider hence the KKT conditions for the optimization problem (\ref{NR}) 
\begin{equation}\label{kkt}
-\mathbf X_{-j}^T(X_j-\mathbf X_{-j}\hat\gamma_j)/n + \lambda_j \hat \kappa_j = 0,
\end{equation}
for $j=1,\dots,p$, where $\hat\kappa_j$ is the sub-differential of the function $\gamma_j \mapsto \|\gamma_j\|_1$ at $\hat\gamma_j,$ i.e. 
\[ 
\hat\kappa_{j,k} = 
\begin{cases}
 \text{sign}(\hat\gamma_{j,k}) & \text{if } \hat\gamma_{j,k}\not = 0\\
 a_{j,k} \in [-1,1] & \text{otherwise},
\end{cases}
\]
where $k\in \{1,\dots,p\}\setminus \{j\}$.
If we define $\hat Z_j$ to be a $p\times 1$ vector 
$$\hat {Z}_j:= ( \hat \kappa_{j,1},\dots,\hat \kappa_{j,j-1},0, \hat \kappa_{j,j+1},\dots,\hat \kappa_{j,p} )/{\hat\tau_j^2},$$
then the KKT conditions may be equivalently stated as follows
\begin{equation}\label{mkkt}
\hat{\Sigma} \hat{\Theta}_j -  e_j - \lambda_j\hat{ Z}_j=0, \text{ for }j=1,\dots,p,
\end{equation}
where $\hat\Sigma = \mathbf X^T \mathbf X/n$ is the sample covariance matrix. 
\noindent
This is shown in Lemma 11 
in the Supplementary material. 
Consequently, the KKT conditions \eqref{mkkt} imply a bound $\|\hat\Sigma \hat\Theta_j - e_j\|_\infty \leq \lambda_j/\hat\tau_j^2$ for each $j=1,\dots,p,$ which will be useful later.
\noindent
Note that the KKT conditions 
may be equivalently summarized in a matrix form as 
$\hat\Sigma \hat\Theta - I -  \hat Z\Lambda=0 ,$
where the columns of $\hat Z$ are given by $\hat Z_j$ for $j=1,\dots,p$ and $\Lambda$ is a diagonal matrix with elements $(\lambda_1,\dots,\lambda_p)$. 
\\
Multiplying the KKT conditions
(\ref{mkkt}) by $\hat\Theta_i$, we obtain 
$$\hat\Theta_i^T (\hat\Sigma \hat\Theta_j - e_j) -  \hat\Theta_i^T\lambda_j\hat Z_j=0.$$
Then we note that adding $\hat\Theta_{ij}-\Theta^0_{ij}$ to both sides 
and rearranging we get
\begin{eqnarray}\label{motiv}
\hat\Theta_{ij} -\hat\Theta_i^T\lambda_j\hat Z_j - \Theta^0_{ij} 
&=& 
\hat\Theta_{ij} -\hat\Theta_i^T (\hat\Sigma \hat\Theta_j - e_j) - \Theta^0_{ij} \\\nonumber
&=&
-(\Theta^0_i)^T (\hat\Sigma-\Sigma_0)\Theta^0_j + \tilde\Delta_{ij},
\end{eqnarray}
where $\tilde \Delta_{ij}=-(\hat\Theta_{i}-\Theta_{i}^0)^T (\hat \Sigma \hat\Theta_j -e_j) - (\hat\Theta_{j}-\Theta^0_j)^T (\hat\Sigma\Theta^0_i-e_i)$ is a term which can be shown to be $o_{\mathbb P}(n^{-1/2})$ under certain conditions (Lemma \ref{rem}).
Hence we define the de-sparsified nodewise Lasso estimator 
\begin{equation}\label{desp}
\hat T:= \hat\Theta - \hat\Theta^T(\hat\Sigma \hat\Theta - I) = \hat\Theta + \hat\Theta^T - \hat\Theta^T \hat\Sigma \hat\Theta.
\end{equation}

\section{Theoretical results}
\label{subsec:main}
\noindent
In this part, we inspect the asymptotic behaviour of the de-sparsified nodewise Lasso estimator (\ref{desp}). In particular we consider the limiting distribution of individual entries of $\hat T$ and show the convergence to the Gaussian distribution is uniform over the considered model. 
For construction of confidence intervals, we consider estimators of the asymptotic variance of the proposed estimator, both for Gaussian and sub-Gaussian design.
We derive convergence rates of the method in supremum norm and consider application to variable selection.
\\
For completeness, in Lemma \ref{rates} in the Supplementary material, we summarize convergence rates of the nodewise Lasso estimator. The result is essentially the same as Theorem 2.4 in \cite{vdgeer13} so the proof of the common parts is omitted.
Recall that  $\mathcal V=\{1,\dots,p\}$ and we define 
the row sparsity by $s_j := \|\Theta_j^0\|_0$, maximum row sparsity by $s:= \max_{1 \leq j\leq p} s_j$ and the coordinates of non-zero entries of the precision matrix by $S_0:= \{(i,j)\in \mathcal V\times \mathcal V :\Theta^0_{ij} \not = 0\}$. 
For the analysis below, we will need the following conditions.
\begin{enumerate}[label=A\arabic*,start=1]
\item \label{eig}
(Bounded spectrum) The inverse covariance matrix $\Theta_0:=\Sigma_0^{-1}$ exists
and there exists a universal constant $L\geq 1$ such that 
$$1/L \leq \Lambda_{\min}(\Theta_0) \leq \Lambda_{\max}(\Theta_0) \leq L.$$
\vskip 0.2cm
\item \label{sparsity}
(Sparsity)
$ {\frac{s \log p}{n}} = o(1).$
\vskip 0.2cm
\item 
\label{subgv}
(Sub-Gaussianity condition)
Suppose that the design matrix $\mathbf X$ has uniformly sub-Gaussian rows $X_i$, i.e. 
there exists a universal constant $K$ such that 
$$\sup_{\alpha\in\mathbb R^p:\|\alpha\|_2\leq 1}\mathbb E\exp\left({{|\alpha^T X_i|^2}/{K^2 }}\right) 
\leq 
2  \quad (i=1,\dots,n).$$
\end{enumerate}
The lower bound in \ref{eig} guarantees that the noise level $\tau^2_j=1/\Theta_{jj}^0$ does not diverge.
The upper bound (equivalently lower bound on eigenvalues of $\Sigma_0$) guarantees that the compatibility condition (see \cite{hds})  is satisfied for the  matrix $\Sigma^0_{-j,-j}$, which is the true covariance matrix $\Sigma_0$ without the $j$-th row and $j$-th column. 
The sub-Gaussianity condition  \ref{subgv} is used to obtain concentration results which are crucial to our analysis. 
Condition \ref{subgv} is also used to ensure that the compatibility condition  is satisfied for $\hat \Sigma$ with high probability (see \cite{hds}). 
Conditions \ref{eig},
\ref{sparsity} 
and \ref{subgv} are the same conditions as used in \cite{vdgeer13} to obtain rates of convergence for the nodewise regression estimator. 
Define the parameter set 
$$\mathcal G(s) := 
\{\Theta\in\mathbb R^{p\times p}: \max_{1\leq i\leq p}\|\Theta_i\|_0 \leq s,
\ref{eig} \text{ is satisfied}\}.$$
The following lemma shows that the proposed estimator $\hat T$ can be decomposed into a pivot term and a term which is of small order $1/\sqrt{n}$ with high probability.

\begin{lema} \label{rem} 
Suppose that $\hat\Theta$ is the nodewise Lasso estimator with regularization parameters $\lambda_j \geq c 
\sqrt{\frac{\log p}{n}}$, uniformly in $j$, for some sufficiently large constant $c>0$. Suppose that \ref{sparsity} and \ref{subgv} 
are satisfied. 
Then for each $(i,j)\in \mathcal V\times \mathcal V$ it holds
\begin{equation}\label{main}
\sqrt{n}(\hat T_{ij}-\Theta^0_{ij}) = 
- \sqrt{n}(\Theta^0_{i})^T(\hat\Sigma -\Sigma_0)\Theta^0_{j} + {\Delta_{ij}},
\end{equation}
where there exists a constant $C>0$ such that
$$\lim_{n\rightarrow \infty }\sup_{\Theta_0\in\mathcal G(s)} \mathbb P\left(
\max_{i,j=1,\dots,p}|\Delta_{ij}| \geq C \frac{s {\log p}}{\sqrt{n}} \right)=0.$$
\end{lema}

\noindent
From Lemma \ref{rem} it follows that 
we need to assume stronger sparsity condition than \ref{sparsity} for the remainder term $\Delta_{ij}$ to be negligible after normalization by $\sqrt{n}$. This is accordance with other literature on the topic, see \cite{vdgeer13}, \cite{zhou}.
Hence we introduce the following strengthened sparsity condition. 
\begin{enumerate}[label=A\arabic*$^*$,start=2]
\item
\label{sparsity.n}
$\frac{s\log p}{\sqrt{n}} = o(1).$
\end{enumerate}
The next result shows that the elements of $\hat T$ are indeed asymptotically normal.
To this end, we further define the asymptotic variance 
$$\sigma_{ij}^2 := \text{var}({(\Theta^0_{i})}^T X_{1}X_{1}^T {\Theta^0_{j}}).$$ 
In some of the results to follow, we shall assume a universal lower bound on $\sigma_{ij}$ as follows. 
\begin{enumerate}[label=A\arabic*,start=4]
\item\label{pos.var}
There exists a universal constant $\omega >0 $ such that
$\sigma_{ij} \geq \omega.$
\end{enumerate}
Assumption \ref{pos.var} is satisfied e.g. under Gaussian design and \ref{eig}. 
Denote a parameter set 
$$\tilde{\mathcal G}(s) := 
\{\Theta\in\mathbb R^{p\times p}: \max_{1 \leq i \leq p}\|\Theta_i\|_0 \leq s,
\ref{eig}, \ref{pos.var} \text{ are satisfied}
\}.$$
\begin{theorem}[Asymptotic normality] 
\label{an} 
Suppose that $\hat\Theta$ is the nodewise Lasso estimator with regularization parameters $\lambda_j \asymp
\sqrt{\frac{\log p}{n}}$ uniformly in $j$.  
Suppose that  \ref{sparsity.n} and \ref{subgv} are satisfied.
Then for every $(i,j)\in \mathcal V\times \mathcal V$ and $z\in\mathbb R$ it holds
$$\lim_{n\rightarrow \infty }\sup_{\Theta_0  
 \in  \tilde{\mathcal G}  (s)}|
\mathbb P_{\Theta_0}\left( \sqrt{n}{(\hat T_{ij}-\Theta_{ij}^0)}/{\sigma_{ij}} \leq z \right) - \Phi(z)| = 0.$$
\end{theorem}
To construct confidence intervals, 
 a consistent estimator of the asymptotic variance $\sigma_{ij}$ is required. Consistent estimators of $\sigma_{ij}$ are discussed in Section \ref{subsec:est.var} (see  Lemmas \ref{var} and \ref{var.general}).
Hence Theorem \ref{an}  implies uniformly valid asymptotic confidence intervals 
$I_\alpha := [\hat T_{ij}\pm \Phi^{-1}(1-\alpha/2)\hat\sigma_{ij}/\sqrt{n}]$, i.e.
$$\lim_{n\rightarrow \infty }\sup_{\Theta_0  
 \in  \tilde{\mathcal G}  (s)}|
\mathbb P_{\Theta_0}( \Theta^0_{ij}\in I_\alpha) -(1-\alpha)| = 0.$$
The result also enables testing hypotheses about individual elements of the precision matrix. For testing multiple hypothesis simultaneously, we may use the standard procedures such as Bonferroni-Holm procedure (see \cite{vdgeer13}).

\subsection{Variance estimation}
\label{subsec:est.var}
\noindent
For the case of Gaussian observations, we may easily calculate the theoretical variance and plug in the estimate $\hat\Theta$ in place of the unknown $\Theta_0$ as is displayed in Lemma \ref{var} below. 

\begin{lema}\label{var} Suppose that  assumptions  \ref{sparsity} and \ref{subgv} are satisfied and assume that the rows of the design matrix $\mathbf X$ are independent $\mathcal N(0,\Sigma_0)$-distributed. Let $\hat\Theta$ be the nodewise Lasso estimator and let $\lambda_j\geq c\tau \sqrt{\log p/n}$ uniformly in $j$ for some $\tau,c>0.$
Then
for $\hat\sigma^2_{ij}:= \hat\Theta_{ii} \hat\Theta_{jj} + \hat\Theta_{ij}^2$
we have
$$
\sup_{\Theta_0 \in \mathcal G(s)}\mathbb P\left(
\max_{i,j=1,\dots,p}
|\hat\sigma^2_{ij} - \sigma^2_{ij}| \geq C_\tau \sqrt{s\log p/n}\right) \leq c_1 p^{1-\tau c_2},$$
for some constants $C_\tau,c_1,c_2>0.$
\end{lema}
\noindent
Lemma \ref{var} implies that under $s=o(\sqrt{n}/\log p)$, we have a rate
$|\hat\sigma^2_{ij} - \sigma^2_{ij}| = o_{\mathbb P}(1/{n}^{1/4}).$
If Gaussianity is not assumed, we may replace the estimator of the variance with the empirical version, and plug in $\hat\Theta$ in place of the unknown $\Theta_0.$ 
Thus we take the following estimator of $\sigma_{ij}^2$, where $\hat\Theta$ is  the nodewise regression estimator
\begin{equation}\label{var.gene}
\hat\sigma_{ij}^2 := \frac{1}{n}\sum_{k=1}^n (\hat\Theta_i^T X_k X_k^T \hat\Theta_j)^2- \hat\Theta_{ij}^2.
\end{equation}
The following lemma justifies this procedure under \ref{eig}, \ref{sparsity.n}, \ref{subgv}.
\begin{lema}\label{var.general}
Suppose that the assumptions \ref{sparsity.n} and \ref{subgv} are satisfied and for some $\epsilon>0$, it holds that 
$\lim_{n\rightarrow \infty}\log^4 (p\vee n) /n^{1-\epsilon} = 0$. 
Let $\hat\Theta$ be the nodewise Lasso estimator and let $\lambda_j\geq c\tau \sqrt{\log p/n}$ uniformly in $j$ for some $\tau,c>0.$
Let $\hat\sigma_{ij}$ be the estimator defined in (\ref{var.gene}). 
Then for all $\eta>0$
$$\lim_{ n \rightarrow \infty}\sup_{\Theta_0 \in \mathcal G(s)}\mathbb P\left( \max_{i,j=1,\dots,p}
| \hat\sigma_{ij}^2 - \sigma_{ij}^2 | \geq \eta\right)=0.$$
\end{lema}
\noindent

\noindent
\subsection{Rates of convergence}\label{subsec:trates}
The de-sparsified estimator achieves optimal rates of convergence in supremum norm. 
Observe first that for the nodewise regression estimator  it holds by (\ref{motiv}), Lemma 1
and Lemma 10 
in the Supplementary material that
$$
\hat \Theta_{ij} - \Theta_{ij}^0 
=\hat\Theta_i^T(\hat\Sigma\hat\Theta_j-e_j)  + \mathcal O_{\mathbb P}\left(\max\left\{\frac{1}{\sqrt{n}}, s{\frac{\log p}{n}}\right\}\right).
$$
By H\"older's inequality and the KKT conditions it follows
$$|\hat\Theta_i^T (\hat\Sigma\hat\Theta_j-e_j)|
\leq \lambda_j\|\hat\Theta_i\|_1/\hat\tau_j^2 
.$$
Consequently, for the rates of convergence of the nodewise Lasso in supremum norm  we find
$$\|\hat \Theta_{} - \Theta_{0} \|_\infty = \mathcal O_\mathbb P\left( \max\left\{\max_{i,j=1,\dots,p}\lambda_j\|\hat\Theta_i\|_1/\hat\tau_j^2 , \frac{1}{\sqrt{n}}, 
s{\frac{\log p}{n}} \right\}\right).$$
De-sparsifying the estimator $\hat\Theta$ as in (\ref{desp}) removes the term involving $\lambda_j\|\hat\Theta_i\|_1/\hat\tau_j^2$ in the above rates.
\begin{theorem}[Rates of convergence]
\label{trates}
Assume that \ref{sparsity} and \ref{subgv} are satisfied. Let $\tau>0$ and let $\hat T$ be the de-sparsified nodewise Lasso estimator with regularization parameters $\lambda_j\geq c\tau\sqrt{\log p/n}$ for some sufficiently large constant $c>0$, uniformly in $j$. Then there exist constants $C_\tau,c_1,c_2>0$ such that
$$\sup_{\Theta_0 \in \mathcal G(s)}
\mathbb P
\left(
|\hat T_{ij}-\Theta^0_{ij}| 
\geq
C_\tau \max\left\{ \frac{1}{\sqrt{n}},   s\frac{\log p}{n}\right\}\right) \leq c_1e^{-c_2\tau}.
$$
and
$$\sup_{\Theta_0 \in \mathcal G(s)}\mathbb P\left( 
\|\hat T_{}-\Theta_{0}\|_\infty 
\geq
C_\tau\max\left\{\sqrt{\frac{\log p}{{n}}}, s\frac{\log p}{n}\right\}\right) \leq c_1p^{1-c_2\tau}.
$$
\end{theorem}
\noindent
We compare  the results of Theorem \ref{trates} with results  on optimal rates of convergence derived for Gaussian graphical models in \cite{zhou}.
Suppose that the observations are Gaussian, i.e. $X_1,\dots,X_n\sim \mathcal N(0,\Sigma_0)$. 
For
 $s\leq C_0 n/\log p$ for some $C_0>0$ and $p \geq s^\nu\text{ for some } \nu >2$
it holds (see \cite{zhou})
$$\inf_{\hat \Theta_{ij}}\sup_{\Theta_0\in \mathcal G(s)}
\mathbb P\left(
|\hat \Theta_{ij}-\Theta^0_{ij}| >
\max\left\{ C_1\frac{1}{\sqrt{n}} , C_2  s\frac{\log p}{n} \right\}\right) > c_1 >0,
$$
and
$$\inf_{\hat \Theta_{}}\sup_{\Theta_0\in \mathcal G(s)}
\mathbb P\left(
\|\hat \Theta_{}-\Theta_{0}\|_\infty > \max\left\{C_1'\sqrt{\frac{\log p}{{n}}}, C_2's\frac{\log p}{n}\right\}\right) > c_2>0,
$$
where $C_1,C_2,C_1',C_2'$ are positive constants depending on $\nu$ and $C_0$ only.
\noindent
As follows from Theorem \ref{trates}, the de-sparsified nodewise Lasso attains the lower bound on rates and thus is in this sense optimal  
(considering the Gaussian setting).

\subsection{Other de-sparsified estimators}
\label{subsec:other.est}
The de-sparsification may work for other estimators of the precision matrix, provided that certain conditions are satisfied. This is formulated in Lemma \ref{other.est} below. A particular example of interest is the square-root nodewise Lasso estimator, that will be discussed below. This estimator has the advantage that it is self-scaling in the variance, similarly as the square-root Lasso \citep{sqrtlasso} on which it is based. 
\begin{lema}\label{other.est}
Assume that for some estimator $\hat\Omega = (\hat\Omega_1,\dots,\hat\Omega_p)$ 
 it holds 
\begin{equation}\label{con1}
\max_{j=1,\dots,p}\|\hat\Omega_j-\Theta^0_j\|_1 =\mathcal O_{\mathbb P}( s\sqrt{\log p/n}),\;
\|\hat\Sigma\hat \Omega  - I\|_\infty =\mathcal O_{\mathbb P}( \sqrt{\log p/n}).
\end{equation}
Then for $\hat T := \hat\Omega + \hat\Omega^T -\hat\Omega^T \hat\Sigma \hat\Omega$ it holds
under \ref{eig}, \ref{sparsity.n}, \ref{subgv}  
$$\|\hat T_{} - \Theta_{0}\|_\infty = \mathcal O_{\mathbb P}(\max\{s\log p/n,\sqrt{\log p/n}\}).$$ 
Moreover, ${\sqrt{n}(\hat T_{ij}-\Theta^0_{ij})}/\sigma_{ij} \rightsquigarrow \mathcal N(0,1).$
\end{lema}
\noindent
We briefly consider nodewise regression with the square-root Lasso as an example. 
The square-root Lasso estimators may be defined via 
$$\hat\gamma_{j} := \text{arg}\min_{\gamma\in\mathbb R^{p-1}} \|X_j- \mathbf X_{-j}\gamma\|_2/n + 2\lambda_0 \|\gamma\|_1,$$
for $j=1,\dots,p.$
Define 
$\hat\tau_j^2 := \|X_j-\mathbf X_{-j}\hat\gamma_j\|_2^2/n$ and $\tilde\tau_j^2 := \hat\tau_j^2 + \lambda_0\hat\tau_j \|\hat\gamma_j\|_1$.
The nodewise square-root Lasso is then given by $\hat\Theta_{j,\text{sqrt}} := \hat\Gamma_j / \tilde\tau_j^2,$
where 
$$\hat{\Gamma}_j := (-\hat\gamma_{j,1},\dots, -\hat\gamma_{j,j-1},1, -\hat\gamma_{j,j+1} ,\dots,-\hat\gamma_{j,p})^T.$$
Note that compared to the nodewise Lasso, the difference lies in estimation of the partial correlations, where we used the square-root Lasso that ``removes the square'' from the squared loss. 
The Karush-Kuhn-Tucker conditions similarly as in Lemma 11
in Supplementary material give
$$\hat\Sigma \hat\Theta_{j,\text{sqrt}} - e_j - \frac{\hat\tau_j}{\tilde\tau_j^2}\lambda_0 \tilde Z_j =0,$$
where $\tilde Z_j := (\tilde \kappa_{j,1},\dots, \tilde \kappa_{j,j-1},1,\tilde\kappa_{j,j+1},\dots,\tilde \kappa_{j,p})^T$
and $\tilde \kappa_{j,i}$ is the sub-diffe\-ren\-tial of the function $\beta\mapsto \|\beta\|_1$ with respect to $\beta_i,$ evaluated at $\hat \gamma_j.$  
The paper \cite{sqrtlasso} further shows that the $\ell_1$-rates for the square-root Lasso satisfy condition \eqref{con1}. 
The de-sparsified estimator may then be defined in the same way as in (\ref{desp}).
Then the conditions of Lemma \ref{other.est} are satisfied and this implies that a de-sparsified nodewise square-root Lasso achieves the same rates as the de-sparsified nodewise Lasso and thus is also rate-optimal.

\subsection{The thresholded estimator and variable selection}
\label{subsec:varsel}
The de-sparsified estimator can be used for variable selection without imposing irrepresentable
conditions.
Under mild conditions, the procedure 
leads to exact recovery of the coefficients that are sufficiently larger in absolute value than the noise level.
The following corollary is implied by Theorem \ref{trates} and Lemmas \ref{var} and \ref{var.general}. 
\noindent
\begin{cor} \label{varsel}
Let $\hat\Theta$ be obtained using the nodewise Lasso and $\hat T$ be defined as in (\ref{desp}) with tuning parameters $\lambda_j \geq c\tau \sqrt{\log p/n}$ uniformly in $j$, for some $c,\tau>0$. Assume that conditions \ref{eig}, \ref{sparsity.n}, \ref{subgv} and \ref{pos.var} are satisfied. 
Let $\hat\sigma_{ij},i,j=1,\dots,p$  be the estimator from Lemma \ref{var.general} and assume that $\log^4 (p\vee n)/n^{1-\epsilon}=o(1)$ for some $\epsilon>0$. Then there exists some constant $C_\tau>0$ such that
$$\lim_{n\rightarrow \infty}
\mathbb P(\max_{i,j=1,\dots,p}|\hat T_{ij} - \Theta_{ij}^0 |/\hat\sigma_{ij} \geq C_\tau 
\sqrt{\log p/n} 
)=0.$$
If, in addition, the rows of $\mathbf X$ are $\mathcal N(0,\Sigma_0)$-dsitributed and $\hat\sigma_{ij},i,j=1,\dots,p$ is instead the estimator from Lemma \ref{var}, then there exist constants $c_1,c_2,C_\tau$ such that 
$$\mathbb P(\max_{i,j=1,\dots,p}|\hat T_{ij} - \Theta_{ij}^0 |/\hat\sigma_{ij} \geq C_\tau 
\sqrt{\log p/n} 
)\leq c_1p^{1-c_2\tau}
.$$

\end{cor}

\noindent
Corollary \ref{varsel} implies that we may define the re-sparsified estimator
$$\hat T^{\text{thresh}}_{ij}:= \hat T_{ij}\mathbf 1_{|\hat T_{ij}|>C_\tau\hat \sigma_{ij} \sqrt{\log p/n}  },$$
where $\hat\sigma_{ij}$ is defined as in Corollary \ref{varsel}.
Denote $\hat S^{\text{thresh}}:=\{(i,j)\in\mathcal V\times \mathcal V: \hat T^{\text{thresh}}_{ij}\not = 0\}.$
Denote $S_0^{\text{act}} := \{(i,j)\in\mathcal V\times \mathcal V: |\Theta^0_{ij}| \geq 2C_\tau\sigma_{ij} \sqrt{\log p/n}\}.$ 
Then it follows directly from Corollary \ref{varsel} that with high probability
$$S_0^{\text{act}}\subset\hat S^{\text{thresh}}\subset S_0.$$
The inclusion $S_0^{\text{act}}\subset \hat S^{\text{thresh}}$ represents that $\hat T^{\text{thresh}}$ correctly identifies all the non-zero parameters which are above the noise level. The inclusion $\hat S^{\text{thresh}}\subset S_0$ means that there are no false positives. 
If for all $(i,j)\in S_0$ it holds
\begin{equation}\label{uss}
|\Theta_{ij}^0| \geq 2 C_\tau\sigma_{ij} \sqrt{\log p/n},
\end{equation}
then we have exact recovery, i.e. with high probability: 
$\hat S^{\text{thresh}} = S_0.$

\section{Further comparison to previous work}
\label{subsec:comparison}
Closely related is the paper \cite{jvdgeer14}, where asymptotically normal estimation of elements of the concentration matrix is considered based on the graphical Lasso. 
While the analysis follows the same principles, the estimation method used here does not require the irrepresentability condition \citep{ravikumar} that is assumed in \cite{jvdgeer14}. Hence we are able to show that our results hold uniformly over the considered class. 
Furthermore, regarding the computational cost, our method uses Lasso regressions, which can be implemented using fast algorithms as in \cite{lars}.
In comparison, the graphical Lasso method presents a more challenging computational problem, for more details see e.g. \cite{glasso.comp}.
\\
Another related work is the paper \cite{zhou}. This paper suggests an estimator for the precision matrix which is shown to have one dimensional asymptotically normal components with asymptotic variance to $\Theta^0_{ii}\Theta^0_{jj}+(\Theta^0_{ij})^2.$ 
The assumptions and results used in the paper are essentially identical with our assumptions and theoretical results in the present paper.  
However, there are some differences. The paper \cite{zhou} assumes Gaussianity of the underlying distribution, while we only require sub-Gaussianity of the margins. Another difference is in the construction of the estimators. Both approaches use regression to estimate the elements of the precision matrix, but the paper \cite{zhou} concentrates on estimation of the joint distribution of each pair of variables $(X^i,X^j)$ for $i,j=1,\dots,p.$ Thus it is computationally more intensive as it requires $\mathcal O(ps)$ high-dimensional regressions (see \cite{zhou}), while our methodology only requires $\mathcal O(p).$

\section{Simulation results}
\label{subsec:sim}
 In this section we report on the performance of our method on simulated data and provide a comparison to another methodology. 
The random sample $X_1,\dots,X_n$ satisfies $\mathbb E X_i =0, \text{var}(X_i) = \Theta_0^{-1}$,
where the precision matrix  $\Theta_0=\text{five-diag}(\rho_0, \rho_1, \rho_2)$ is defined by 
\[
\Theta^0_{ij} =
\begin{cases}
\rho_0  & \mbox{ if } i=j,\\
\rho_1 & \mbox{ if }|i-j|= 1,\\
\rho_2 & \mbox{ if }|i-j|=2,\\
0 & \mbox{ otherwise. }
\end{cases}
\]
\\
We consider the settings $\mathbf S_1 =(\rho_{0},\rho_{1},\rho_{2}) = (1,0.3,0)$ and $\mathbf S_2 =(\rho_{0},\rho_{1},\rho_{2}) =(1,0.5,0.3)$. The second setting $(1,0.5,0.3)$ is further adjusted by  randomly perturbing each non-zero off-diagonal element of $\Theta_0$ by adding a realization from the uniform distribution on the interval $[-0.05,0.05].$ We denote this new perturbed model by $(1,0.5,0.3)_U.$
Hence the second precision matrix was chosen randomly. 
The sparsity assumption requires $s = o(\sqrt{n}/\log p).$ 
We have chosen the sample sizes for numerical experiments according to the sparsity assumption (for this purpose, we ignored possible constants in the sparsity restriction), i.e. $n \geq s^2\log^2 p$.

\subsection{Asymptotic normality and confidence intervals for individual parameters}
\label{subsec:ae}
\subsubsection{The Gaussian setting}
\label{subsec:gauss}
In this section, we consider normally-distributed observations, $X^i\sim \mathcal N(0,\Theta_0^{-1}),$
for $i=1,\dots,n.$
In Figure 1
 we display histograms of $\sqrt{n}(\hat T_{ij}-\Theta^0_{ij})/\hat\sigma_{ij}$ for  $(i,j)\in\{(1,1),(1,2),$ $(1,3)\},$ where $\hat T$ is defined in (\ref{desp}) and the empirical variance $\hat \sigma_{ij}$ is estimated as suggested by  Lemma \ref{var}. 
Superimposed is the density of $\mathcal N(0,1).$\\
Secondly, we investigate the properties of confidence intervals constructed using the de-sparsified nodewise Lasso. For comparison, we also provide results using confidence intervals based on the de-sparsified graphical Lasso  introduced in \cite{jvdgeer14}.  
 The coverage and length of the confidence interval were
 estimated by their empirical versions, 
$$\hat\alpha_{ij} := \mathbb P_N \mathbf 1_{\{\Theta_{0,ij} \in  I_{ij,\alpha}\}} \text{ and }
\hat \ell_{ij} := \mathbb P_N 2\Phi^{-1}(1-\alpha/2){\hat\sigma_{ij}}/{\sqrt{n}},$$
respectively, using $N=300$ random samples.
For a set $A\subset V\times V$, we define the average coverage over the set $A$ (and analogously average length $\text{avglength}_{A}$) as
$$\text{avgcov}_A := \sum_{(i,j)\in A} \hat\alpha_{ij}/{|A|}.$$
We report average coverages over the sets $S_0$ and $S_0^c$. These are denoted by $\text{avgcov}_{S_0}$ and $\text{avgcov}_{S_0^c},$ respectively. 
Similarly, we calculate average lengths of confidence intervals for each parameter $\Theta^0_{ij}$ from $N=300$ iterations and report $\text{avglength}_{S_0}$ and $\text{avglength}_{S_0^c}$.
\noindent\\
The results of the simulations are shown in Tables 1 and 2.
 The target coverage level is $95\%.$
The methodology for the choice of the tuning parameters was used as follows (see \cite{zhou}), for both methods, 
\begin{equation}\label{tp}
\hat s =\sqrt{n}/\log p, B=\text{qt}(1-\hat s /(2p),n-1), \lambda =B/\sqrt{n-1+B^2},
\end{equation}
where $\text{qt}(\beta,n-1)$ denotes the $\beta$-quantile of a $t$-distribution with $n-1$ degrees of freedom.

\begin{figure}[h!]\label{fig:hist}
\footnotesize
\centering
\textbf{Asymptotic normality in the Gaussian setting}\\
\includegraphics[width=\textwidth]{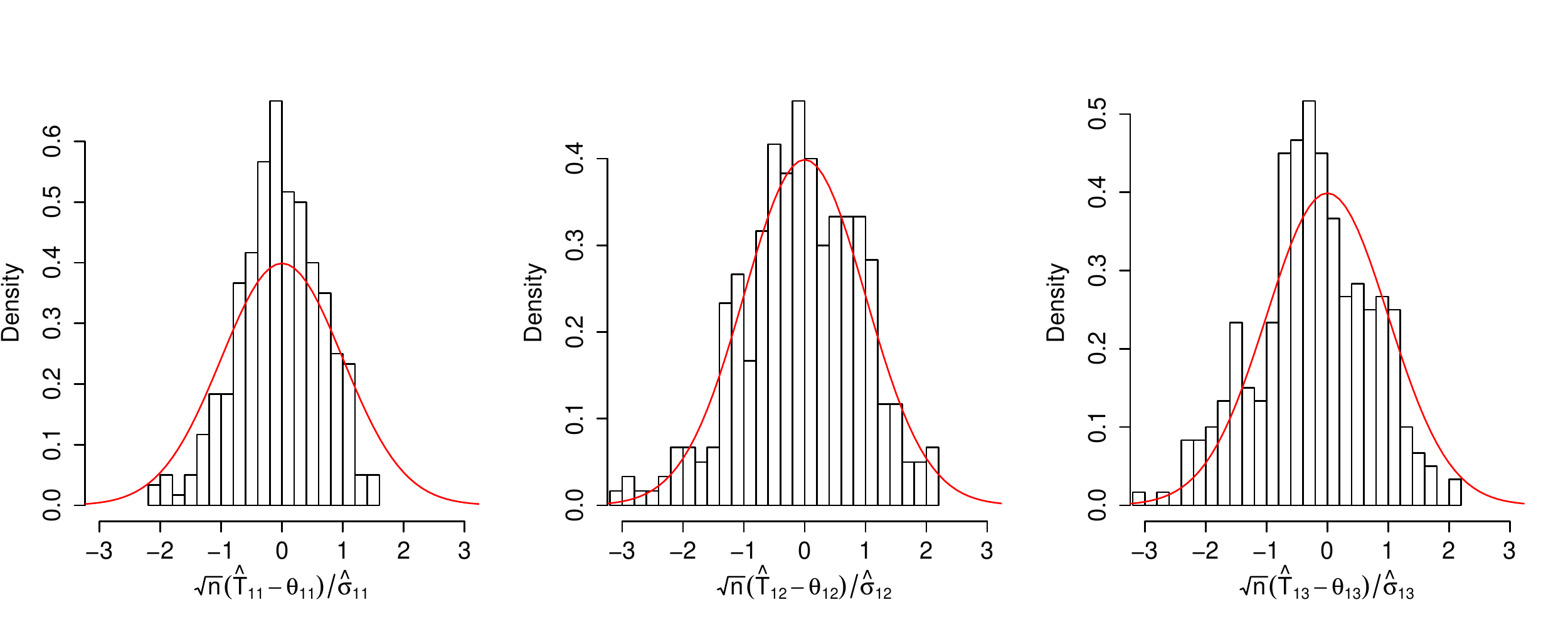}
\caption{\footnotesize
Histograms for $\sqrt{n}(\hat T_{ij}-\Theta^0_{ij})/\hat\sigma_{ij}$, $(i,j)\in\{(1,1),(1,2),(1,3)\}$. The sample size was $n=500$ and the number of parameters $p=100$. The nodewise regression estimator was calculated $300$ times. 
The setting is $\mathbf S_1=(1,0.3,0)$.}
\end{figure}

\begin{table}[h!] \label{fig:cov}
\begin{center}
\footnotesize
\textbf{Gaussian setting: Estimated coverage probabilities and lengths}\par\medskip
\begin{tabular}{cclcccc}\hline
   \multicolumn{3}{c}{\multirow{2}{*}{Setting $\mathbf S_1=(1,0.3,0)$}}        &  &  &  &  \\
   \multicolumn{2}{c}{}     &  & $S_0$ & $S_0$ & $S_0^c$ & $S_0^c$  \\
   $p$& {  $n$ } && avgcov & avglength & avgcov & avglength \\[1ex]
\hline\\[-2ex]
\multirow{2}{*}{$100$} &\multirow{2}{*}{ $191$} 
& D-S NW &  0.945 & 0.302 & 0.963 & 0.262 \\ 
&& D-S GL & 0.931 & 0.293 & 0.974 & 0.254 
\\[0.07cm]\hline \\[-2ex]
\multirow{2}{*}{$200$} &\multirow{2}{*}{ $253$} 
& D-S NW   & 0.947 & 0.267 & 0.963 & 0.232 \\ 
&&  D-S GL & 0.928 & 0.254 & 0.976 & 0.220 
\\[0.07cm]\hline \\[-2ex]
\multirow{2}{*}{$300$} &\multirow{2}{*}{ $293$} 
& D-S NW & 0.949 & 0.238 & 0.965 & 0.220 \\ 
&&  D-S GL  & 0.928 & 0.236 & 0.977 & 0.205 
 \\[0.07cm]\hline\\[-2ex]
\multirow{2}{*}{$400$} &\multirow{2}{*}{ $324$} 
&  D-S NW & 0.948 & 0.246 & 0.965 & 0.230 \\ 
&&  D-S GL  & 0.925 & 0.228 & 0.981 & 0.223 
\\[0.07cm]\hline 
\end{tabular}
\caption{\footnotesize A table showing a comparison of de-sparsified nodewise Lasso (D-S NW) and de-sparsified graphical Lasso (D-S GL).
Parameter $p$ takes values $100,200,300,400$  and the corresponding values $n$ are given by $n=\lceil s^2\log ^2 p\rceil$, where $s=3$. 
The regularization parameter was chosen as described in \eqref{tp}. The number of generated random samples was $N=300.$
}
\end{center}
\end{table}

\begin{table}[h!] \label{fig:cov}
\begin{center}
\footnotesize
\textbf{Gaussian setting: Estimated coverage probabilities and lengths}\par\medskip
\begin{tabular}{cclcccc}\hline
   \multicolumn{3}{c}{\multirow{2}{*}{Setting $\mathbf S_2=(1,0.5,0.3)_U$}}     &     &  &  &  \\
   \multicolumn{2}{c}{}     &  & $S_0$ & $S_0$ & $S_0^c$ & $S_0^c$  \\
   $p$& {  $n$ } && avgcov & avglength & avgcov & avglength \\[1ex]
\hline\\[-2ex]
\multirow{2}{*}{$100$} &\multirow{2}{*}{ $531$} 
& D-S NW & 0.896 & 0.164 & 0.975 & 0.146 \\ 
 && D-S GL & 0.781 & 0.153 & 0.980 & 0.137 
\\[0.07cm]
   \hline\\[-2ex]
\multirow{2}{*}{$200$} &\multirow{2}{*}{ $702$} 
& D-S NW & 0.868 & 0.142 & 0.976 & 0.126 \\ 
&& D-S GL & 0.729 & 0.133 & 0.982 & 0.119 
\\[0.07cm] \hline\\[-2ex]
\multirow{2}{*}{$300$} &\multirow{2}{*}{ $814$} 
& D-S NW & 0.863 & 0.131 & 0.976 & 0.117 \\ 
&& D-S GL & 0.712 & 0.124 & 0.984 & 0.110 
\\[0.07cm]    \hline\\[-2ex]
\multirow{2}{*}{$400$} &\multirow{2}{*}{ $898$} 
& D-S NW & 0.859 & 0.125 & 0.976 & 0.111 \\ 
&& D-S GL & 0.709 & 0.118 & 0.984 & 0.105 
\\[0.07cm]\hline 
\end{tabular}
\caption{\footnotesize 
A table showing a comparison of de-sparsified nodewise Lasso (D-S NW) and the de-sparsified graphical Lasso (D-S GL).
Parameter $p$ takes values $100,200,300,400$  and the corresponding values $n$ are given by $n=\lceil s^2\log ^2 p\rceil$, where $s=5$. 
The regularization parameter was chosen as described in \eqref{tp}. The number of generated random samples was $N=300.$
}
\end{center}
\end{table}

\subsubsection{A sub-Gaussian setting}
\label{subsec:subgauss}
In this section, we consider a design matrix with rows having a sub-Gaussian distribution other than the Gaussian distribution.
Let $U:=(U_1,\dots,U_n)$ be an $n\times p$ matrix with jointly independent entries generated from a continuous uniform distribution on the interval $[-\sqrt{3},\sqrt{3}]$. 
Further consider a matrix $ \Theta_0:= \text{five-diag}(1,0.3,0)$ and let $\Sigma_0=\Theta_0^{-1}.$
Then we define 
$$X_i:= \Sigma_0^{1/2} U_i$$ for 
$i=1,\dots,n.$
 Then the expectation of $X_i$ is zero and the covariance matrix of $X_i$ is exactly  $\Sigma_0$ and the precision matrix is $\Theta_0.$ 
It follows by Hoeffding's inequality that $X_i$ defined as above is sub-Gaussian with a universal constant $K>0.$
\\
A further difference compared to the simulations in Section \ref{subsec:gauss} is that we now estimate the variance of the de-sparsified estimator using the formula proposed in \eqref{var.gene} for sub-Gaussian settings:
\begin{equation}
\hat\sigma_{ij}^2 := \frac{1}{n}\sum_{k=1}^n (\hat\Theta_i^T X_k X_k^T \hat\Theta_j)^2- \hat\Theta_{ij}^2,
\end{equation}
where $\hat \Theta$ is the nodewise Lasso.
The regularization parameters for the nodewise Lasso are used in accordance with \eqref{tp}.
Figure 2 again displays the histograms related to several entries of the de-sparsified nodewise Lasso.
Results related to the constructed confidence intervals are summarized in Table 3.
The results demonstrate that the de-sparsified nodewise Lasso performs relatively well even under this  non-Gaussian setting.

\begin{figure}[h!]\label{fig:hist}
\footnotesize
\centering
\textbf{Asymptotic normality in the sub-Gaussian setting}\\
\includegraphics[width=\textwidth]{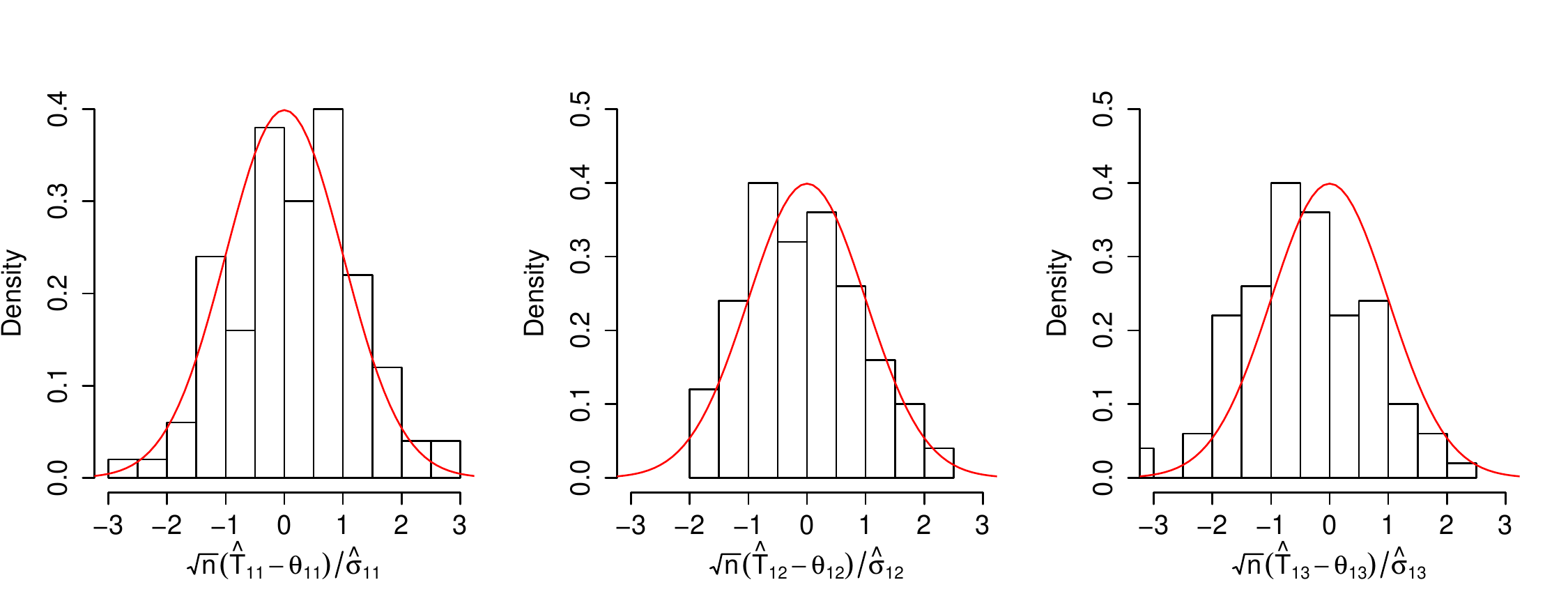}
\caption{\footnotesize
Histograms for $\sqrt{n}(\hat T_{ij}-\Theta^0_{ij})/\hat\sigma_{ij}$, $(i,j)\in\{(1,1),(1,2),(1,3)\}$. The sample size was $n=500$ and the number of parameters $p=100$. The nodewise regression estimator was calculated $300$ times. 
The setting  is $\mathbf S_1=(1,0.3,0)$.}
\end{figure}

\begin{table}[h!] \label{fig:cov_sg}
\begin{center}
\footnotesize
\textbf{Sub-Gaussian setting: Estimated coverage probabilities and lengths}\par\medskip
\begin{tabular}{cclcccc}\hline
   \multicolumn{3}{c}{\multirow{2}{*}{Setting $\mathbf S_1=(1,0.3,0)$}}     &     &  &  &  \\
   \multicolumn{2}{c}{}    &   & $S_0$ & $S_0$ & $S_0^c$ & $S_0^c$  \\
   $p$& {  $n$ } && avgcov & avglength & avgcov & avglength \\[1ex]
\hline\\[-2ex]
\multirow{2}{*}{$100$} &\multirow{2}{*}{ $191$} 
& D-S NW & 0.906 & 0.234 & 0.949 & 0.249 
\\ 
 && D-S GL & 0.811 & 0.190 & 0.944 & 0.216
\\[0.07cm]
   \hline
	\\[-2ex]
\multirow{2}{*}{$200$} &\multirow{2}{*}{ $253$} 
& D-S NW
 & 0.909 & 0.203 & 0.950 & 0.217
 \\ 
&& D-S GL & 0.791 & 0.165 & 0.946 & 0.187
\\[0.07cm]
 \hline
\\[-2ex]
\multirow{2}{*}{$300$} &\multirow{2}{*}{ $293$} 
& D-S NW 
& 0.911 & 0.189 & 0.950 & 0.202
 \\ 
&& D-S GL &0.765 & 0.152 & 0.947 & 0.173
\\[0.07cm]   
 \hline
\\[-2ex]
\multirow{2}{*}{$400$} &\multirow{2}{*}{ $324$} 
& D-S NW 
& 0.911 & 0.180 & 0.951 & 0.192
\\ 
&& D-S GL & 0.740 & 0.143 & 0.947 & 0.164
\\[0.07cm]
\hline 
\end{tabular}
\caption{\footnotesize 
A table showing a comparison of de-sparsified nodewise Lasso (D-S NW) and the de-sparsified graphical Lasso (D-S GL).
Parameter $p$ takes values $100,200,300,400$  and the corresponding values $n$ are given by $n=\lceil s^2\log ^2 p\rceil$, where $s=3$. 
The regularization parameter was chosen as described in \eqref{tp}. The number of generated random samples was $N=300.$
}
\end{center}
\end{table}

\subsection{Variable selection}
\label{subsec:vs}
For variable selection as suggested in Corollary \ref{varsel}, we compare the de-sparsified nodewise Lasso and the de-sparsified graphical Lasso. The setting is again as in Section \ref{subsec:gauss}.
Average true positives and false positives over 100 repetitions are reported. 
Choice of the tuning parameters is according to (\ref{tp}) and the thresholding level is given by
\begin{equation}\label{tl}
\lambda_{\text{thresh}} = \hat\sigma_{ij}\sqrt{2\nu \frac{\log p}{n}},
\end{equation}
taking $\nu=1$ for the de-sparsified nodewise regression, $\nu=0.5$ for the de-sparsified graphical Lasso. We take $\hat\sigma_{ij}=\hat\Theta_{ii}\hat\Theta_{jj}+\hat\Theta_{ij}^2$ as in Lemma \ref{var}.
The results of this simulation experiment are summarized in Table 4.

\begin{table}[ht]
\footnotesize
\centering
\textbf{Estimated true positives (TP) and false positives (FP)}\par\medskip
\begin{tabular}{llcccc}
  \hline     \\[-2ex]
  \multicolumn{2}{c}{{Setting $\mathbf S_1 =(1,0.5,0.4)$}} & TP & TP rate \% & FP & FP rate \% \\[0.1cm] 
 \hline\\[-2ex]
    $p=100$              &  D-S NW  & 494 &  100.0 & 0 & 0 \\ 
   $|S_0|=494$           &  D-S GL & 493.98 & 99.999 & 0 & 0 \\[0.07cm]\hline \\[-2ex]
      $p=200$            &  D-S NW    & 994 & 100.0 & 0 & 0 \\ 
    $|S_0|=994$          &  D-S GL         & 993.62  &  99.961 & 0 & 0 \\[0.07cm]\hline \\[-2ex]
    $p=300$              &  D-S NW         & 1494 & 100.0 & 0 & 0\\ 
$|S_0|=1494$ & D-S GL                      &  1492.42 & 99.894 & 0 & 0 \\[0.07cm]\hline \\[-2ex]
 $p=400$   & D-S NW                       & 1994.00 & 100.0 & 0 & 0 \\ 
$|S_0|=1994$ & D-S GL                     & 1989.08 & 99.753 & 0 & 0 \\[0.07cm]
   \hline
\end{tabular}
\caption{\footnotesize
Estimated true positives (TP) and false positives (FP) for the de-sparsified nodewise regression estimator (D-S NW)
and for the D-S GL estimator. The sample size $n=400$ was held constant for all the values of $p;$ the number of repetitions was $N=100$. The thresholding levels was chosen as in (\ref{tl}). }
\end{table}

\section{Real data experiments}
\label{subsec:real}
We consider two real datasets, where we model the conditional independence structure of the covariates using a graphical model. In particular, we aim to do edge selection and we estimate the edge structure of the graphical model using the de-sparsified nodewise Lasso.
The first dataset is the Prostate Tumor Gene Expression dataset, which is available in the \texttt{R} package \texttt{spls}.
The second dataset is about
riboflavin (vitamin $B_2$) production by bacillus subtilis. 
The dataset is available from the \texttt{R} package \texttt{hdi}.
\\
For both datasets, the procedure is essentially identical. We only consider the first $500$ covariates which have the highest variances. 
In the first step, we split the sample and use 10 randomly chosen observations to estimate the variances of the $500$ variables. With the estimated variances, we scale the design matrix containing the remaining observations. 
We calculate the nodewise Lasso using the tuning parameter as in the simulation study,
and then calculate the de-sparsified nodewise Lasso. We threshold the de-sparsified nodewise Lasso at the  level
$\Phi^{-1}(1-\alpha/(2p^2)) \hat\sigma_{ij}/\sqrt{n},$
 where $\alpha=0.05$ and $\hat\sigma_{ij}^2= \hat\Theta_{ii}\hat\Theta_{ii}+\hat\Theta_{ij}^2$ is an estimate of the asymptotic variance calculated under the assumption of normality and using the nodewise Lasso estimator $\hat\Theta.$
\\
The first dataset contained observations of $p=4088$ logarithms of genes expression levels from $n=71$ genetically engineered mutants of bacillus subtilis. We considered 500 variables with the highest variances, hence a full graph contains $ {500}\choose {2}$ edges. 
The de-sparsified nodewise Lasso  identified $20$ edges as significant. For comparison, the de-sparsified graphical Lasso introduced in \cite{jvdgeer14} identified $5$ edges as significant. It is worth pointing out that the set of edges selected by the de-sparsified graphical Lasso is a subset of the edges selected by de-sparsified nodewise Lasso.\\
The second dataset contained $n=102$ observations on $p=6033$ variables. We used the procedure above to do edge selection using the de-sparsified nodewise Lasso. Our analysis identified $108$ edges as significant using the de-sparsified nodewise Lasso. 
For comparison, the de-sparsified graphical Lasso identified $28$ edges as significant. Again, the set of edges selected by the de-sparsified graphical Lasso is a subset of the edges selected by de-sparsified nodewise Lasso.

\section{Conclusions}
\label{subsec:conclusions}
We proposed a methodology for low-dimensional inference in high-dimensional graphical models. The method, called the de-sparsified nodewise Lasso, is easy to implement and computationally competitive with the state-of-art methods. We studied asymptotic properties of the de-sparsified nodewise Lasso under mild conditions on the model. The de-sparsified nodewise Lasso enjoys rate optimality in supremum norm and leads to exact variable selection under beta-min conditions and mild conditions on the model. We demonstrated its performance on several models in a simulation study and on two real datasets. These numerical studies showed that it performs well in a variety of settings, including non-Gaussian settings. Further open questions concern for instance the asymptotic efficiency of the proposed estimator, similarly as in the low dimensional settings.




\bibliography{gminf}

\newpage
%

\section*{Supplementary material}
We summarize some preliminary material in Section \ref{sec:preli} and proofs of all the results in Section \ref{sec:ma}. In appendices \ref{sec:concentration} and \ref{sec:rates2}, we summarize some known results.






\section{Preliminary results: Rates of convergence of the nodewise Lasso}
\label{sec:preli}
We provide a brief  overview of the results for the nodewise Lasso (derived in \cite{vdgeer13})
in the following two Lemmas. Lemma \ref{T2} below follows from classical concentration results for sub-Gaussian random variables (see \cite{hds}).
The proof of Lemma \ref{rates} below can be found in \cite{vdgeer13}.


\begin{lema}\label{T2}
Suppose that $X_j, j=1,\dots,p$, with $\mathbb EX_j=0$ and $\emph{var}(X_j)=\Sigma_0$ satisfies the eigenvalue condition \ref{eig} and the sub-Gaussianity condition \ref{subgv}. Then there exist constants $c_1,c_2,c>0$
such that for any $\tau$ sufficiently large 
$$\mathbb P(\mathcal T_j^c ) \leq c_1p^{-\tau c_2},$$
where 
\begin{eqnarray*}
\mathcal T_j&:=&\{ 
\|\hat\Sigma\Theta_j^0 - e_j\|_\infty /n \leq c\tau \sqrt{\log p /n}, \\
&& \;\; \|\hat\Sigma-\Sigma_0\|_\infty \leq c\tau\sqrt{\log p/n}, 
\\
&& \;\;|\eta_j^T \eta_j/n -\tau_j^2| \leq c\tau \sqrt{\log p/n}\}.  
\end{eqnarray*}
Moreover, for the set $\cap_{j=1}^p \mathcal T_j$ we get by the union bound
$$\mathbb P((\cap_{j=1}^p \mathcal T_j)^c) = \mathbb P(\cup_{j=1}^p \mathcal T_j^c) \leq p \max_{j=1,\dots,p} \mathbb P(\mathcal T_j^c) \leq 
c_1  p^{1-\tau c_2}.
$$
\end{lema}

\noindent
\begin{lema}[a version of Theorem 2.4 in \cite{vdgeer13}] \label{rates}
Let $\tau>0$ and suppose that $\hat\Theta$ is the nodewise Lasso estimator (\ref{nrdef}) with regularization parameters $\lambda_j\geq 
 c\tau\sqrt{\frac{\log p}{n}}$ uniformly in $j$, for some sufficiently large constant $c>0$. 
Suppose that \ref{eig}, \ref{sparsity}, \ref{subgv} are satisfied. 
Then there exists a constant $C_\tau>0$ such
that on the set $\mathcal T_j$ defined in Lemma \ref{T2} it holds
$$\|\hat\Theta_j-\Theta_j^0\|_1 \leq  C_\tau s \sqrt{{\log p}/{n}}, \;|\hat\tau_j^2 - \tau_j^2| \leq  C_\tau \sqrt{{s\log p}/{n}},$$
$$\|\hat\Theta_j-\Theta_j^0\|_2^2\leq  C_\tau s {{\log p}/{n}}
.$$
Furthermore, on the set $\cap_{j=1}^p \mathcal T_j$ we have 
$$\max_{j=1,\dots,p}\|\hat\Theta_j - \Theta_j^0\|_1 \leq C_\tau s\sqrt{\log p/n}, \quad\quad\max_{j=1,\dots,p}|\hat\tau_j^2 - \tau_j^2| \leq  C_\tau \sqrt{
{s\log p}/{n}},$$
$$\max_{j=1,\dots,p}\|\hat\Theta_j-\Theta_j^0\|_2^2\leq  C_\tau s {{\log p}/{n}}.$$
\end{lema}

\section{Proofs for Section \ref{subsec:main}}

\label{sec:ma}

\begin{proof}[Lemma \ref{rem}]
We first derive a bound for $|\Delta_{ij}|$ (which will be useful later) and then a bound for $\|\Delta\|_\infty.$\\
Let $\tau>0$ and 
consider the set $\mathcal T_j$ defined in Lemma \ref{T2}.
Note that by Lemma \ref{T2}  (and since $\|\Theta_i^0\|_2 \leq L,\|\Theta_j^0\|_2 \leq L $) we have 
$\mathbb P(\mathcal T_i \cap \mathcal T_j ) \leq c_1 p^{-c_2\tau}$ for some constants $c_1,c_2>0.$
We now condition on the set $\mathcal T_i \cap \mathcal T_j $.
\\
By the definition of $\hat T,$ we have
\begin{eqnarray*}
\hat T_{ij} - \Theta_{ij}^0 &=& 
\hat\Theta_{ij}-\Theta_{ij}^0  - \hat\Theta_i^T\lambda_j\hat Z_j    \\
& = &  
- ({\Theta_i^0})^T(\hat\Sigma-\Sigma^0)\Theta_j^0 
\\&&
+ \underbrace{(\hat\Sigma\Theta_i^0 - e_i)^T(\Theta_j^0 -\hat\Theta_j)}_{\text{rem}_1}   
+ \underbrace{(\Theta_i^0-\hat\Theta_i)^T(\hat\Sigma\hat\Theta_j - e_j)}_{\text{rem}_2}.
\end{eqnarray*}
For the first remainder, we obtain
$$|\text{rem}_1| = |(\hat\Sigma\Theta_i^0 - e_i)^T(\Theta_j^0 -\hat\Theta_j)| 
\leq 
\|\hat\Sigma\Theta_i^0 - e_i\|_\infty \|\Theta_j^0 -\hat\Theta_j\|_1.
$$
Under \ref{eig} and \ref{subgv}, by Lemma \ref{conc} in Appendix \ref{sec:concentration},  we have 
$$\max_{j=1,\dots,p}\|\hat \Sigma \Theta_j^0 - e_j\|_\infty =\mathcal O_{\mathbb P}(\sqrt{\log p/n}).$$
Consequently, and using $\ell_1$-rates of convergence from Lemma \ref{rates}, 
we obtain for some constant $C_\tau>0$
$$|\text{rem}_1| \leq  C_\tau s{\log p}/{n}.$$
By Lemma \ref{rates}, we have $|\hat\tau_j^2 - \tau_j^2| \leq  C_\tau\sqrt{\frac{s\log p}{n}}$ and 
since $1/\tau_j^2=\mathcal O(1)$, we have 
$1/\hat\tau_j^2 =\mathcal O(C_\tau) $.
Hence for the second remainder, by the KKT conditions and Lemma \ref{rates}, we obtain
\begin{eqnarray*}
|\text{rem}_2| &=& |(\Theta^0_i-\hat\Theta_i)^T(\hat\Sigma\hat\Theta_j - e_j)|\\
& \leq & 
\|\hat\Sigma\hat\Theta_j - e_j\|_\infty \|\Theta^0_i-\hat\Theta_i\|_1
\\
& \leq &
\lambda_j/\hat\tau_j^2  (1+\|\hat\kappa_j\|_1) s\sqrt{\log p/n} =\mathcal O( C_\tau s \log p/n).
\end{eqnarray*}
Therefore, for $\tilde\Delta_{ij}:= \text{rem}_1 +\text{rem}_2,$ 
we have $ |\tilde\Delta_{ij}|= \mathcal O_{}( s{\log p}/{n})$ on the set $\mathcal T_i \cap \mathcal T_j$. 
\\
\\
Now condition on the set $\cap_{j=1}^p \mathcal T_j.$ By Lemma \ref{T2}, we have for some constants $c_1,c_2>0$ it holds
$$\mathbb P((\cap_{j=1}^p \mathcal T_j )^c) \leq c_1p^{1-c_2\tau}.$$
We again have the decomposition
\begin{equation}\label{hlavna}
\hat T - \Theta_0  = 
- \Theta_0(\hat\Sigma-\Sigma_0) \Theta_0
+ \underbrace{
(\Theta_0\hat\Sigma -I)(\Theta_0- \hat\Theta)
+ 
(\Theta_0-\hat\Theta)^T(\hat\Sigma \hat\Theta -I)}_{
\tilde\Delta
}.
\end{equation}
But then we obtain
\begin{eqnarray*}
\|\tilde\Delta\|_\infty & \leq & 
\|(\Theta_0\hat\Sigma -I)(\Theta_0- \hat\Theta)\|_\infty +
\|(\Theta_0-\hat\Theta)^T(\hat\Sigma \hat\Theta -I)\|_\infty
\\
&\leq &
 \|\Theta_0\hat\Sigma -I\|_\infty \vertiii{\Theta_0- \hat\Theta}_1 +
\vertiii{\Theta_0-\hat\Theta}_1 \|\hat\Sigma \hat\Theta -I\|_\infty
\\
&\leq & 
\max_{j=1,\dots,p} \|\hat\Sigma\Theta_j^0 -e_j\|_\infty \vertiii{\Theta_0- \hat\Theta}_1 \\
&&+
\vertiii{\Theta_0- \hat\Theta}_1 \max_{j=1,\dots,p} \lambda_j (1+\|\hat\kappa_j\|_1)/\hat\tau_j^2.
\end{eqnarray*}
By Lemma \ref{rates}, we have $\vertiii{\Theta_0- \hat\Theta}_1 =\mathcal O(C_\tau s\sqrt{\log p/n})$ and $\max_{j=1,\dots,p}|\hat\tau_j^2 -\tau_j^2|=\mathcal O(C_\tau \max_{j=1,\dots,p}\sqrt{s}\lambda_j)$. Hence, and by the last display 
 we get 
 \begin{eqnarray*}
\|\Delta\|_\infty =
\sqrt{n}\|\tilde\Delta\|_\infty  = \mathcal O(C_\tau s\log p/\sqrt{n}). 
\end{eqnarray*}
Therefore for some constant $C_\tau>0$ we obtain 
$$\mathbb P(\|\Delta\|_\infty   \geq C_\tau s\log p/\sqrt{n} ) \leq \mathbb P((\cap_{j=1}^p \mathcal T_j )^c) \leq c_1p^{1-c_2\tau}.$$
Now note that the constants $c_1,c_2$ are universal, therefore we can take the supremum over $\mathcal G(s)$ to obtain
$$\sup_{\Theta_0\in\mathcal G(s)}\mathbb P(\|\Delta\|_\infty   \geq C_\tau s\log p/\sqrt{n} ) \leq  c_1p^{1-c_2\tau}.$$
If we choose $\tau>0$ sufficiently large so that $1-c_2\tau <0$ then 
$$\lim_{n\rightarrow \infty}\sup_{\Theta_0\in\mathcal G(s)}\mathbb P(\|\Delta\|_\infty   \geq C_\tau  s\log p/\sqrt{n} ) \leq \lim_{n\rightarrow \infty} c_1p^{1-c_2\tau} = 0.$$

\end{proof}


\begin{proof}[Theorem \ref{an}]
By (\ref{main}), it holds
$$\sqrt{n}(\hat T_{ij}-\Theta^0_{ij})/\sigma_{ij} = -\tilde Z_{}/\sigma_{ij} + \text{rem}/\sigma_{ij},$$
where 
$$\tilde Z_{} := \sqrt{n}{(\Theta^0_{i})^T (\hat\Sigma-\Sigma_0)\Theta^0_{j}}
=:
\sqrt{n}\sum_{k=1}^n  \frac{1}{n}((\Theta^0_{i})^T X_{k}X_k^T\Theta^0_{j}-\Theta^0_{ij}),$$ 
and
$\text{rem} := \sqrt{n} (\hat\Sigma\Theta_i^0 - e_i)^T ( \Theta_j^0 - \hat\Theta_j^0) + 
\sqrt{n}(\Theta_i^0 - \hat\Theta_i^0 )^T (\hat\Sigma\hat\Theta_j^0 - e_j).$
\\
First observe that by Lemma \ref{rem},
$|\text{rem}| =\mathcal O_{\mathbb P}\left( s\frac{\log p}{\sqrt{n}}\right) = o_{\mathbb P}(1).$
Note that under \ref{subgv}, the results of Lemma \ref{rem} are uniform in $\Theta_0.$ 
\vskip 0.2cm
\noindent
Denote $Z_{k}:= (\Theta^0_{i})^T X_{k}X_k^T\Theta^0_{j}-\Theta^0_{ij}.$
To show $\tilde Z_{}/\sigma_{ij}\rightsquigarrow \mathcal N(0,1)$, we apply Berry-Esseen theorem (see e.g. \cite{durrett}).
By Lemma \ref{bilinear} in Appendix \ref{sec:concentration} (note that $\|\Theta_i^0\|_2\leq L,\|\Theta_j^0\|_2\leq L$), we have
a moment bound 
for $m\geq 2,$
$$\mathbb E|Z_{k}|^m/(2\sigma_0 L^2K)^{m} 
\leq 
\frac{m!}{2}  (K/\sigma_0)^{m-2},
$$
for some constant $\sigma_0>0.$
Under \ref{eig}, for every $m$ fixed we have $\sup_{\tilde{\mathcal G}}\mathbb E|Z_{k}|^m=\mathcal O(1)$. 
We have $\sigma^2_{ij}=\mathbb E(Z_{1}^2) \geq \omega^2$ for some universal constant $\omega>0$. 
Therefore,
$$|\mathbb P_{\Theta_0}(\tilde Z_{}/\sigma_{ij}<z) - \Phi(z)| \leq \frac{3\mathbb E|Z_{1}|^3}{(\mathbb E|Z_{1}|^2)^{3/2}\sqrt{n}} \leq \frac{C}{\sqrt{n}},$$
where $C$ does not depend on $\Theta_0$ and $n$,
which concludes the proof.
\vskip 0.2cm
\noindent
\end{proof}

\subsection{Proofs for Section \ref{subsec:est.var}}

\begin{proof}[Lemma \ref{var}]
Since $X_1\sim \mathcal N(0,\Sigma_0),$ then $Z:={\Theta_0} X\sim \mathcal N(0,\Theta_0)$. It is well known that  
$\mathbb E(Z_i^2Z_j^2) = \Theta^0_{ii}\Theta^0_{jj} + 2 (\Theta^0_{ij})^2.$
Consequently, we get
$$\sigma_{ij}^2 = \text{var}(({\Theta^0_{i}})^T X_1{(\Theta^0_{j})}^T X_1) 
=\text{var}(({\Theta_0} X_1)_i({\Theta_0} X_1)_j)
= \Theta^0_{ii}\Theta^0_{jj}+{(\Theta^0_{ij})^2}.$$
We next use the results of Lemmas \ref{T2} and \ref{rates}. The set $\cap_{i=1}^p\mathcal T_j$ from Lemma \ref{T2} holds with probability at least $1-c_1p^{1-\tau c_2}$, for some constants $c_1,c_2>0$. 
We define $\hat \Delta_{ij}:=\hat\Theta_{ij}-\Theta_{ij}^0.$
Conditioning on the set $\cap_{i=1}^p\mathcal T_j$  we thus get
(using Lemma \ref{rates} (Euclidean norm bound) and condition \ref{eig})
\begin{eqnarray}\nonumber
\max_{i,j=1,\dots,p}|\hat\sigma_{ij}^2 - \sigma_{ij}^2|
& \leq &
\max_{i,j=1,\dots,p}|\hat\Theta_{ii}\hat\Theta_{jj}- \Theta^0_{ii}\Theta^0_{jj}| + 
\max_{i,j=1,\dots,p}|\hat\Theta_{ij}^2  - ({\Theta}^0_{ij})^2 | 
\\\nonumber
&\leq & \max_{i,j=1,\dots,p}|\hat\Delta_{ii}\hat\Delta_{jj} + \Theta^0_{ii}\hat\Delta_{jj}+\Theta^0_{jj}\hat\Delta_{ii}| 
\\\nonumber
&&+ \;
\max_{i,j=1,\dots,p}|\hat\Delta_{ij}(\hat\Delta_{ij} + 2\Theta^0_{ij})|
\\
\label{cons}
&\leq&   
C_\tau\sqrt{s{\log p}/{n}},
\end{eqnarray}
for some constant $C_\tau>0.$
\end{proof}


\begin{proof}[Lemma \ref{var.general}]
To simplify notation, we write $\Theta:=\Theta_0.$  
We have
\begin{eqnarray*}
\max_{i,j=1,\dots,p}|\hat\sigma_{ij}^2 -\sigma_{ij}^2| &= &
\max_{i,j=1,\dots,p}|\left\{\frac{1}{n}\sum_{k=1}^n (\hat\Theta_i^T X_k X_k^T \hat\Theta_j)^2- \hat\Theta_{ij}^2 \right\}
\\&& -
\left\{\mathbb E(\Theta_i X_1X_1^T \Theta_j)^2 - \Theta_{ij}^2\right\}|
\\
&\leq &
\max_{i,j=1,\dots,p}|\underbrace{\frac{1}{n}\sum_{k=1}^n (\hat\Theta_i^T X_k X_k^T \hat\Theta_j)^2
-
\mathbb E(\Theta_i X_1X_1^T \Theta_j)^2| }_{I}
\\
&&+\max_{i,j=1,\dots,p}|\underbrace{
\hat\Theta_{ij}^2 - \Theta_{ij}^2 
}_{II}|.
\end{eqnarray*}
By \ref{eig}, $|\Theta_{ij}|=\mathcal O(1)$ and by Lemma \ref{rates},  we have
$$\max_{i,j=1,\dots,p}\|\hat\Theta_{i}-\Theta_i\|_2 = \mathcal O_{\mathbb P}(\sqrt{s{\log p}/{n}}).$$
Hence
\begin{eqnarray*}
\max_{i,j=1,\dots,p}|II| 
\leq
2\max_{i,j=1,\dots,p}|\Theta_{ij}||\hat\Theta_{ij} - \Theta_{ij}| + |\hat\Theta_{ij} - \Theta_{ij}|^2
 = \mathcal O_{\mathbb P}(\sqrt{s{\log p}/{n}}).
\end{eqnarray*}

\noindent
For the first term, we have
\begin{eqnarray*}
\max_{i,j=1,\dots,p}|I|&=&\max_{i,j=1,\dots,p}|\frac{1}{n}\sum_{k=1}^n (\hat\Theta_i^T X_k X_k^T \hat\Theta_j)^2
-
\mathbb E(\Theta_i X_1X_1^T \Theta_j)^2 |\\
&\leq & 
\max_{i,j=1,\dots,p}\underbrace{
|\frac{1}{n}\sum_{k=1}^n (\hat\Theta_i^T X_k X_k^T \hat\Theta_j)^2 - (\Theta_i^T X_k  X_k^T\Theta_j )^2|
}_{\text{rem}_1}
\\
&&+\underbrace{\max_{i,j=1,\dots,p}
| \frac{1}{n}\sum_{k=1}^n (\Theta_i^T X_k  X_k^T\Theta_j )^2 
-  \mathbb E(\Theta_i X_kX_k^T \Theta_j)^2
|}_{\text{rem}_{2}}.
\end{eqnarray*}

\noindent
By symmetrization,
\begin{eqnarray*}
\mathbb E |\text{rem}_{2}| &\leq & 
2\mathbb E\max_{i,j=1,\dots,p} |\frac{1}{n}\sum_{k=1}^n (\Theta_i^T X_k  X_k^T\Theta_j )^2 \epsilon_k|,
\end{eqnarray*}
where $\epsilon_k,k=1,\dots,n$ is a sequence of independent Radechamer random variables, independent of $\mathbf X$.
Let $K>0$ and consider  truncation by $K$ as follows
\begin{eqnarray*}
\mathbb E|\text{rem}_{2}| &\leq & 
2\mathbb E\underbrace{
\max_{i,j=1,\dots,p}\lvert \frac{1}{n}\sum_{k=1}^n (\Theta_i^T X_k  X_k^T\Theta_j )^2 \epsilon_k 
1_{    (\Theta_i^T X_k  X_k^T\Theta_j )^2\leq K }\rvert
}_{r_1}\\
&&+
2\mathbb E\underbrace{
\max_{i,j=1,\dots,p}|\frac{1}{n}\sum_{k=1}^n (\Theta_i^T X_k  X_k^T\Theta_j )^2 \epsilon_k 
1_{(\Theta_i^T X_k  X_k^T\Theta_j )^2> K }|
}_{r_2}
,
\end{eqnarray*}
The term $r_1$ can be bounded using Hoeffding's inequality since for $i,j=1,\dots,p$ and $k=1,\dots,n$ it holds
$$|(\Theta_i^T X_k  X_k^T\Theta_j )^2 \epsilon_k 
1_{(\Theta_i^T X_k  X_k^T\Theta_j )^2\leq K }|\leq K.$$
 Thus for  any $a>0$
\begin{eqnarray*}
&& \mathbb P(
\max_{i,j=1,\dots,p}|\frac{1}{n}\sum_{k=1}^n (\Theta_i^T X_k  X_k^T\Theta_j )^2 \epsilon_k 
1_{(\Theta_i^T X_k  X_k^T\Theta_j )^2\leq K}| \geq a) 
\\
&& \leq e^{-na^2 / (2K^2)}.
\end{eqnarray*}
Hence integrating over $a,$ we get for the expectation
$$\mathbb E r_1 = \int_0^\infty \mathbb P(r_1>a) da \leq \frac{\sqrt{2\pi}}{2}\frac{K}{\sqrt{n}}.$$
Now we bound $\mathbb E r_2.$
Rewriting the expectation and using the union bound, we obtain
\begin{eqnarray*}
\mathbb E r_2  &= & \int_0^{\infty} \mathbb P(r_2 >x)dx 
\\
&\leq & p^2 n \int_0^{\infty} \max_{i,j=1,\dots,p}\max_{k=1,\dots,n} \mathbb P( |(\Theta_i^T X_k  X_k^T\Theta_j )^2 \epsilon_k
1_{(\Theta_i^T X_k  X_k^T\Theta_j )^2> K }| > x) dx
\\
&\leq &p^2 n
\int_0^{\infty} \max_{i,j=1,\dots,p}\max_{k=1,\dots,n} \mathbb P( (\Theta_i^T X_k  X_k^T\Theta_j )^2 > \max(x,K) ) dx
\end{eqnarray*}
Next by the sub-Gaussianity condition \ref{subgv}, we have
\begin{eqnarray*}
\mathbb P \left(
(\Theta_i^T X_k  X_k^T\Theta_j )^2  
>\max(x,K)
\right)
&\leq &
 c_1 e^{-c_2 \max(x,K)^{1/2}},
\end{eqnarray*}
for some constants $c_1,c_2>0.$ 
Hence
\begin{eqnarray*}
 \mathbb E r_2 & \leq & p^2 n\int_0^{\infty}  c_1 e^{-c_2 \max(x,K)^{1/2}} dx
\\
&=& p^2n c_1 Ke^{-c_2 K^{1/2}}
 + p^2n c_1\frac{2}{c_2}\left(K^{1/2}+\frac{1}{c_2}\right)e^{-c_2 K^{1/2}}.
\end{eqnarray*}
We choose $K = n^{(1-\epsilon)/2}$ where the $\epsilon$ is taken from the assumptions of this lemma.
Then we get
\begin{eqnarray*}
 \mathbb E |\text{rem}_{2}| &\leq & \mathbb E r_1 +\mathbb E r_2 \\
& \leq & 
\frac{\sqrt{2\pi}}{2}n^{-\epsilon/2} 
+
p^2n c_1 n^{(1-\epsilon)/2}e^{-c_2 n^{(1-\epsilon)/4}}
 \\
&&+\; p^2n c_1\frac{2}{c_2}\left(n^{(1-\epsilon)/4}+\frac{1}{c_2}\right)e^{-c_2 n^{(1-\epsilon)/4}}.
\end{eqnarray*}
Then by the assumption $\log (p\vee n) /n^{(1-\epsilon)/4} = o(1),$ we obtain that the above bound converges to zero for $n\rightarrow \infty.$
This implies that $|\text{rem}_{2}|=o_{\mathbb P}(1).$
\\
%
\noindent
To bound $\max_{i,j=1,\dots,p}|\text{rem}_{1}|,$ we denote $W_{i,k}:= \Theta_i^T X_k$ and $\hat W_{i,k} := \hat\Theta_i^T X_k$ for $i=1,\dots,p$ and $k=1,\dots,n$.
Then we can rewrite
\begin{eqnarray*}
\text{rem}_{1}&=&
|\frac{1}{n}\sum_{k=1}^n 
(\hat W_{i,k} \hat W_{j,k})^2 - (W_{i,k}W_{j,k} )^2|.
\end{eqnarray*}
Observe that 
\begin{eqnarray*} \label{expan}
\text{rem}_{1}&=&
|\frac{1}{n}\sum_{k=1}^n 
(\hat W_{i,k} \hat W_{j,k})^2 - (W_{i,k}W_{j,k} )^2 |\\
&=& 
|\frac{1}{n}\sum_{k=1}^n 
 \{ (\hat W_{i,k}-W_{i,k}) (\hat W_{j,k}-W_{j,k}) \\\nonumber
&& \; + \;
(\hat W_{i,k}-W_{i,k}) W_{j,k}  \\\nonumber
&&\; + \; W_{i,k} (\hat W_{j,k} - W_{j,k}) + W_{i,k}W_{j,k} \}^2 
\\&&
- (W_{i,k}W_{j,k} )^2|. 
\end{eqnarray*}
To bound $|\text{rem}_{1}|,$ we make use of the last expression where we expand the curly bracket $\{\cdot\}^2$. 
We have
\begin{eqnarray*} \label{expanded}
\text{rem}_{1}&=&
\frac{1}{n}\sum_{k=1}^n 
 \{ (\hat W_{i,k}-W_{i,k}) (\hat W_{j,k}-W_{j,k}) \; + \;
(\hat W_{i,k}-W_{i,k}) W_{j,k}  \\\nonumber
&&\; + \; W_{i,k} (\hat W_{j,k} - W_{j,k}) + W_{i,k}W_{j,k} \}^2 
- (W_{i,k}W_{j,k} )^2\\\nonumber
&=& \frac{1}{n}\sum_{k=1}^n 
(\hat W_{i,k}-W_{i,k})^2 (\hat W_{j,k}-W_{j,k})^2  \\
\nonumber
&&+\;
2(\hat W_{i,k}-W_{i,k})^2(\hat W_{j,k}-W_{j,k}) W_{j,k}\\\nonumber
&&  +
2 (\hat W_{i,k}-W_{i,k})(\hat W_{j,k}-W_{j,k})^2 W_{i,k} 
\\\nonumber
&&
 +\;4(\hat W_{i,k}-W_{i,k}) (\hat W_{j,k}-W_{j,k})W_{i,k}W_{j,k}
\\\nonumber
&&+(\hat W_{i,k}-W_{i,k})^2 W_{j,k}^2 
+
2(\hat W_{i,k}-W_{i,k}) W^2_{j,k}  W_{i,k}\\\nonumber
&&+W_{i,k}^2 (\hat W_{j,k} - W_{j,k})^2 + 2W^2_{i,k} (\hat W_{j,k} - W_{j,k})W_{j,k}.
\end{eqnarray*}
\noindent
By (iteratively) applying the Cauchy-Schwarz (C-S) inequality to each summation term in the last display, we can bound each term by a term involving $\frac{1}{n}\sum_{k=1}^n (\hat W_{l,k}-W_{l,k})^4$ and $\frac{1}{n}\sum_{k=1}^n W_{l,k}^4$, $l=i,j$.
Consequently, it suffices to find a rate for 
$$\max_{l=1,\dots,p}\frac{1}{n}\sum_{k=1}^n (\hat W_{l,k}-W_{l,k})^4$$ and to show that $
\max_{l=1,\dots,p}\frac{1}{n}\sum_{k=1}^n W_{l,k}^4=\mathcal O_{\mathbb P}(1).$ Then 
$$\max_{i,j=1,\dots,p}|\text{rem}_{1}|=\mathcal O_{\mathbb P}(\max_{l=1,\dots,p}(\frac{1}{n}\sum_{k=1}^n (\hat W_{l,k}-W_{l,k})^4)^{1/4})).$$
We now show that 
\begin{equation}
\label{prva1}
\max_{l=1,\dots,p}\frac{1}{n}\sum_{k=1}^n (\hat W_{l,k}-W_{l,k})^4 = \mathcal O_{\mathbb P}((s\log p /\sqrt{n})^2),
\end{equation} 
and  that
\begin{equation}
\label{druha}
\max_{l=1,\dots,p}\frac{1}{n}\sum_{k=1}^n W_{l,k}^4= \mathcal O_{\mathbb P}(1).
\end{equation}
The claim (\ref{prva1}) follows by Lemma \ref{var.general.aux} below.
We now show (\ref{druha}).
When bounding $\max_{i,j=1,\dots,p}|\text{rem}_{2}|$ above, we have shown that for all $a>0$ 
$$\lim_{n\rightarrow \infty}\mathbb P(\max_{i,j=1,\dots,p}|\frac{1}{n}\sum_{k=1}^n W_{i,k}^2W_{j,k}^2-\mathbb EW_{i,k}^2W_{j,k}^2|\geq a)
= 
0.$$
But then since $\mathbb EW_{i,k}^2W_{j,k}^2= \mathcal O(1),$ it follows that also (note here that the case when $i=j$ is covered)
$$\lim_{n\rightarrow \infty}\mathbb P(\max_{i,j=1,\dots,p}\frac{1}{n}\sum_{k=1}^n W_{i,k}^2W_{j,k}^2 \geq C_2) =0,$$
for some constant $C_2>0.$\\ 
Collecting all the results above,  for any $\eta>0$ we have
$$\lim_{n\rightarrow \infty}\mathbb P(\max_{i,j=1,\dots,p}|\hat\sigma_{ij}^2 -\sigma_{ij}^2|\geq \eta) =0.$$
 
\end{proof}

\begin{lema}\label{var.general.aux}
Under the assumptions of Lemma \ref{var.general} (and using notation of Lemma \ref{var.general}), on the set $\cap_{j=1}^p\mathcal T_j$ it holds
$$\max_{l=1,\dots,p}\frac{1}{n}\sum_{k=1}^n (\hat W_{l,k}-W_{l,k})^4 = \mathcal O_{}((s\log p /\sqrt{n})^2).$$
\end{lema}
\begin{proof}
We will first show that 
$$\max_{l=1,\dots,p}\frac{1}{n}\sum_{k=1}^n (\hat W_{l,k}-W_{l,k})^2 = \mathcal O_{}(s\log p /{n}).$$
To see this, observe that for  $l\in\{i,j\},$
\begin{eqnarray*}
\max_{l=1,\dots,p}\frac{1}{n} \sum_{k=1}^n (\hat W_{l,k}-W_{l,k})^2 
&=& \max_{l=1,\dots,p}\|\mathbf X(\hat\Theta_l-\Theta_l)\|_2^2/n,
\\
&\leq &\max_{l=1,\dots,p}
(\hat\Theta_l-\Theta_l)^T \Sigma_0 (\hat\Theta_l-\Theta_l) \\
&&
+\;\max_{l=1,\dots,p} |(\hat\Theta_l-\Theta_l)^T (\hat\Sigma -\Sigma_0)(\hat\Theta_l-\Theta_l)|
\\
&\leq &
\Lambda_{\max}(\Sigma_0) \max_{l=1,\dots,p}\|\hat\Theta_l-\Theta_l\|_2^2 
\\
&&+\; 
\|\hat\Sigma -\Sigma_0\|_\infty \max_{l=1,\dots,p} \|\hat\Theta_l-\Theta_l\|_1^2.
\end{eqnarray*}
By the above bound and by Lemma \ref{rates}, 
we observe
$$\max_{l=1,\dots,p}\frac{1}{n}\sum_{k=1}^n (\hat W_{l,k}-W_{l,k})^2 = \mathcal O_{\mathbb P}(s\log p /{n}),$$ 
as required. Rewriting this as
$$\max_{l=1,\dots,p}\frac{1}{\sqrt{n}}\sum_{k=1}^n (\hat W_{l,k}-W_{l,k})^2 = \mathcal O_{}(s\log p /\sqrt{n}),$$ and taking square of both sides, we obtain
\begin{eqnarray*}
\max_{l=1,\dots,p}\frac{1}{n}\sum_{k=1}^n (\hat W_{l,k}-W_{l,k})^4 
&\leq &
\max_{l=1,\dots,p}\left(\frac{1}{\sqrt{n}}\sum_{k=1}^n (\hat W_{l,k}-W_{l,k})^2\right)^2\\
&=& 
\mathcal O_{}((s\log p /\sqrt{n})^2).
\end{eqnarray*}

\end{proof}



\subsection{Proofs for Section \ref{subsec:trates}}

\begin{proof}[Theorem \ref{trates}]
By the decomposition (\ref{hlavna}), we have
$$
\hat T - \Theta_0  = 
- \Theta_0(\hat\Sigma-\Sigma_0) \Theta_0
+\underbrace{ {(\Theta_0\hat\Sigma -I)(\Theta_0- \hat\Theta)}_{} 
+ {(\Theta_0-\hat\Theta)^T(\hat\Sigma \hat\Theta -I)}}_{\tilde \Delta}.
$$
Under \ref{eig}, \ref{sparsity}, \ref{subgv} 
by the proof of Lemma \ref{rem}, there exist $c_1,c_2$ such that with probability at least 
$c_1p^{1-c_2\tau}$ it holds that
$$\|\tilde\Delta\|_\infty = \mathcal O_{}(C_\tau {s\log p/n}).$$
By Lemma \ref{prel} in Appendix \ref{sec:concentration}, with probability at least $c_1e^{-c_2 \tau }$ it holds
$$|(\Theta^0_{i})^T (\hat\Sigma-\Sigma_0)\Theta^0_{j}| =\mathcal O_{}(C_\tau/\sqrt{n}).$$
Hence there exists a constant $C_\tau>0$ such that
$$\mathbb P\left(|\hat T_{ij}-\Theta^0_{ij}| 
\geq C_\tau  \max\left\{ \frac{1}{\sqrt{n}},   s\frac{\log p}{n}\right\}\right)
\leq c_3 \max\{p^{1-c_2\tau}, e^{-c_2 \tau }\}.
$$
But then since the constants $c_2,c_3$ are universal, we can take the supremum over the model.
By taking $\tau$ sufficiently large it follows that 
$$\sup_{\Theta_0\in\mathcal G(s)} \mathbb P\left(|\hat T_{ij}-\Theta^0_{ij}| 
\geq C_\tau  \max\left\{ \frac{1}{\sqrt{n}},   s\frac{\log p}{n}\right\}\right)
\leq c_4 e^{-\tau}.$$
We proceed to show the second result of the Theorem. We have 
by Lemma \ref{conc} in Appendix \ref{sec:concentration} that there exist constants $c_1,c_2$ such that with probability at least
$c_1p^{-c_2\tau}$ it holds 
$$\|\Theta_0(\hat\Sigma-\Sigma_0) \Theta_0\|_\infty =\mathcal O_{}(C_\tau\sqrt{\log p/n}).$$ 
We have $\|\hat T_{}-\Theta_{0}\|_\infty  \leq \|\tilde\Delta\|_\infty + \|\Theta_0(\hat\Sigma-\Sigma_0) \Theta_0\|_\infty.$
But then there exist constants $C_\tau,c_1,c_2>0$ such that
$$\mathbb P\left(\|\hat T_{}-\Theta_{0}\|_\infty 
\geq C_\tau\max\left\{\sqrt{\frac{\log p}{{n}}}, s\frac{\log p}{n}\right\}\right) \leq c_1 p^{1-c_2\tau}.
$$

\end{proof}


\subsection{Proofs for Section \ref{subsec:other.est}}

\begin{proof}[Lemma \ref{other.est}]
First note that
\begin{eqnarray*}
\hat T - \Theta_0 
&=&
\hat\Omega- \Theta_0 -\hat\Omega^T(\hat\Sigma\hat\Omega - I)
\\ 
&=&
\hat\Omega- \Theta_0 - \Theta_0 ^T (\hat\Sigma\hat\Omega - I) -(\hat\Omega-\Theta_0)^T (\hat\Sigma\hat\Omega - I)
\\
&=&
 -\Theta_0 (\hat\Sigma-\Sigma_0)\Theta_0 -\underbrace{ (\Theta_0\hat\Sigma - I) (\hat\Omega-\Theta_0)}_{rem_1}-
\underbrace{(\hat\Omega-\Theta_0)^T (\hat\Sigma\hat\Omega - I)}_{rem_2}.
\end{eqnarray*}
Under \ref{eig} and \ref{subgv}  we have  
$$\max_{j=1,\dots,p}\|\hat \Sigma \Theta_j^0 - e_j\|_\infty =\mathcal O_{\mathbb P}(1/\sqrt{n}).
$$
For the first remainder we then have
\begin{eqnarray*}
\|rem_1\|_\infty &=& \|(\Theta_0\hat\Sigma - I) (\hat\Omega-\Theta_0)\|_\infty \leq \| \Theta_0\hat\Sigma - I\|_\infty \vertiii{\Theta_0 - \hat\Omega}_1\\
&=&\mathcal O_{\mathbb P}(s\log p/n).
\end{eqnarray*}
For the second remainder we have
\begin{eqnarray*}
\|rem_2\|_\infty &=&
\|(\hat\Omega-\Theta_0)^T (\hat\Sigma\hat\Omega - I)\|_\infty \leq \| \hat\Sigma\hat\Omega - I\|_\infty \vertiii{\hat\Omega-\Theta_0}_1
\\
&=&\mathcal O_{\mathbb P}(s\log p/n).
\end{eqnarray*}
Collecting the above statements yields for $i,j=1,\dots,p$
\begin{eqnarray*}
\hat T _{ij}- \Theta^0_{ij} 
&=&(\Theta^0_i)^T(\hat\Sigma-\Sigma_0)\Theta^0_j + 
\mathcal O_{\mathbb P} (s\log p/n)\\
&=&(\Theta^0_i)^T(\hat\Sigma-\Sigma_0)\Theta^0_j +o_{\mathbb P}(1/\sqrt{n}).
\end{eqnarray*}
and hence we have
$$\sqrt{n}(\hat T_{ij}-\Theta_{ij}^0)/\sigma_{ij}\rightsquigarrow \mathcal N(0,1).$$
\end{proof}



\appendix

\section{Concentration results for sub-Gaussian design }
\label{sec:concentration}
\begin{lema}\label{bilinear}
Let $\alpha,\beta\in\mathbb R^p$ such that $\|\alpha\|_2 \leq M,\|\beta\|_2\leq M.$
Let $X_k\in\mathbb R^p$ satisfy the sub-Gaussianity assumption \ref{subgv} 
with a constant $K>0.$ 
Then for $m\geq 2,$
$$\mathbb E|\alpha ^T X_k X_k^T\beta - \mathbb E \alpha ^T X_k X_k^T\beta |^m / (2 M^2K^2)^{m} \leq 
\frac{m!}{2}.
$$
\end{lema}
\noindent
\begin{lema}\label{prel}
Let $\alpha,\beta\in\mathbb R^p$ such that $\|\alpha\|_2 \leq M,\|\beta\|_2\leq M.$
Let $X_k\in\mathbb R^p$ satisfy the sub-Gaussianity assumption \ref{subgv} 
with a constant $ K>0.$ 
For all $t>0$ 
$$\mathbb P \left( |\alpha^T \hat\Sigma \beta - \alpha^T \Sigma_0 \beta | /(2M^2K^2)
>
t + \sqrt{2t}
\right) \leq 2e^{-nt}.$$
\end{lema}
\noindent
\begin{lema} \label{conc}
Assume $\|\alpha_i\|_2 \leq M,\|\beta\|_2\leq M$ for all $i=1,\dots,p$ and \ref{subgv} with $K$.
For all $t>0$  it holds
$$\mathbb P 
\left( \max_{i=1,\dots,p} |\alpha_i^T(\hat\Sigma -\Sigma_0) \beta  | / (2M^2K^2)
>
t + \sqrt{2t} + \sqrt{\frac{2\log (2p)}{n}} +\frac{\log (2p)}{n}
\right) \leq e^{-nt}.$$
\end{lema}
For the proofs of Lemmas \ref{bilinear}, \ref{prel}, \ref{conc}, see \cite{jvdgeer14}.

\section{Rates of convergence of the nodewise Lasso}
\label{sec:rates2}

\begin{lema}
\label{modkkt}
Let $\hat\Theta_j$ be obtained as in \eqref{nrdef}.  
Then it holds
$\hat\Sigma \hat\Theta_j - e_j - \lambda_j\hat Z_j=0.$
\end{lema}


\begin{proof}[Lemma \ref{modkkt}]
The KKT  conditions (\ref{kkt}) give 
\begin{equation}\label{kkt.again}
-\mathbf X_{-j}^T(X_j-\mathbf X_{-j}\hat\gamma_j)/n + \lambda_j \hat \kappa_j = 0.
\end{equation}
Multiplying (\ref{kkt.again}) by $\hat\gamma_j^T$ and 
since $\hat\gamma_j^T \hat \kappa_j = \|\hat\gamma_j\|_1$, we obtain 
\begin{equation}\label{tau}
\hat\tau_j^2 = X_{j}^T(X_j - \mathbf X_{-j}\hat\gamma_j)/n.
\end{equation}
Rewriting $X_j-\mathbf X_{-j}\hat\gamma_j = \mathbf X \hat\Gamma_j= \mathbf X \hat\Theta_j \hat\tau_j^2,$ from (\ref{tau}) we obtain 
\begin{equation}\label{one}
X_{j}^T \mathbf X\hat\Theta_{j}/n = 1.
\end{equation}
Similarly rewriting (\ref{kkt.again}) gives
\begin{equation}\label{two}
\mathbf X_{-j}^T\mathbf X\hat\Theta_j/n = \frac{\lambda_j}{\hat\tau_j^2}\kappa_j.
\end{equation}
Combining (\ref{one}) and (\ref{two}) we obtain $\mathbf X^T \mathbf X \hat\Theta_j /n - e_j = \lambda_j\hat Z_j,$
and rearranging, we get
$\hat\Sigma \hat\Theta_j - e_j - \lambda_j\hat Z_j=0$ as required.
\end{proof}

\end{document}